\documentclass[12pt]{amsart} 
\usepackage{verbatim} 
\usepackage{amssymb} 
\usepackage{amsbsy} 
\usepackage{amscd} 
\usepackage{amsmath} 
\usepackage{amsthm} 
\usepackage[mathscr]{eucal} 
\usepackage{graphicx}
 
\usepackage{rotating}
\usepackage{multirow,colordvi}

\numberwithin{equation}{section}
 
\newtheorem{theorem}{Theorem} 
\newtheorem{lemma}[theorem]{Lemma} 
\newtheorem{corollary}[theorem]{Corollary} 
\newtheorem{proposition}[theorem]{Proposition}

\theoremstyle{remark}
\newtheorem{example}[theorem]{Example} 
\newtheorem{remark}[theorem]{Remark}

\newcounter{FNC}[page]
\def\fauxfootnote#1{{\addtocounter{FNC}{2}\Magenta{$^\fnsymbol{FNC}$}%
     \let\thefootnote\relax\footnotetext{\Magenta{$^\fnsymbol{FNC}$#1}}}}
\renewcommand{\qed}{\hfill\raisebox{-5.5pt}{\includegraphics[height=15pt]{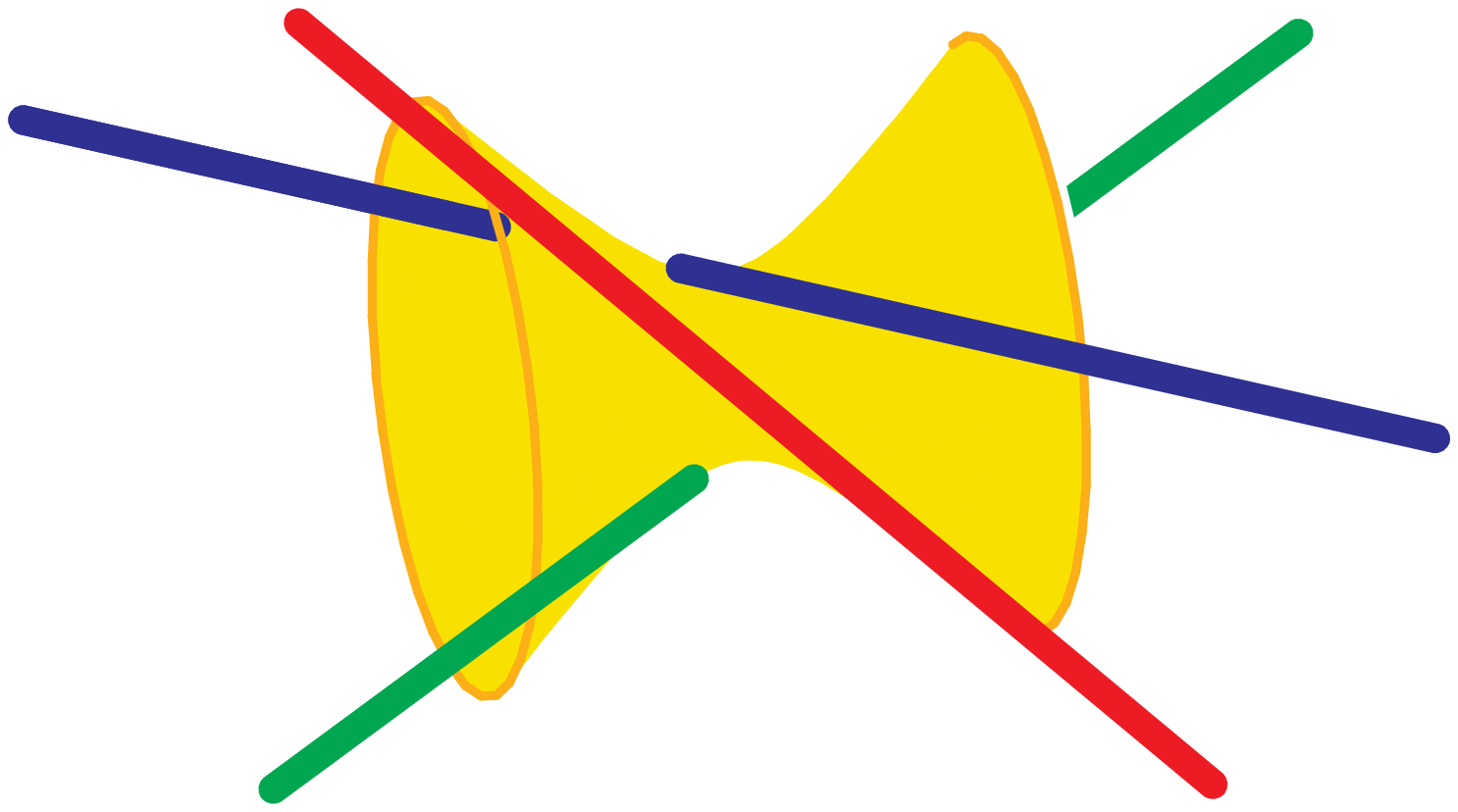}}}

\newcommand{\lhra}{\ensuremath{\lhook\joinrel\relbar\joinrel\relbar\joinrel\rightarrow}}

\DeclareMathOperator{\Wr}{\rm Wr}
\newcommand{\LG}{\mbox{\rm LG}}
\newcommand{\Gr}{\mbox{\rm Gr}}
\DeclareMathOperator{\YT}{\rm YT}

\renewcommand{\P}{{\mathbb P}}
\newcommand{\C}{{\mathbb C}}
\newcommand{\Q}{{\mathbb Q}}
\newcommand{\R}{{\mathbb R}}
\newcommand{\Z}{{\mathbb Z}}

\newcommand{\calT}{{\mathcal T}}

\newcommand{\calG}{{\mathcal G}}
\newcommand{\Gal}[1]{{\mathcal G}_{#1}}

\newcommand{\blambda}{\boldsymbol{\lambda}}
\newcommand{\bb}{{\mathbf b}}

\newcommand{\Fdot}{F_{\bullet}}
\newcommand{\Edot}{E_{\bullet}}
\newcommand{\adot}{{\mathbf a}_{\bullet}}
\newcommand{\alphdot}{{\boldsymbol{\alpha}}_{\bullet}}
\newcommand{\Fl}{{\mathbb F}\ell}

\DeclareMathOperator{\Span}{span }
\DeclareMathOperator{\rank}{rank }

\def\rgbColor#1#2{#2}

\renewcommand{\Green}[1]{\rgbColor{0.133 0.545 0.133}{#1}}

\newcommand{\defcolor}[1]{\Blue{#1}}
\newcommand{\demph}[1]{\defcolor{{\sl #1}}}

\DeclareRobustCommand{\Is}{\includegraphics{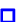}}
\DeclareRobustCommand{\I}{\includegraphics{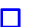}}
\DeclareRobustCommand{\IIs}{\includegraphics{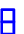}}

\DeclareRobustCommand{\III}{\includegraphics{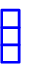}}
\DeclareRobustCommand{\IIIs}{\includegraphics{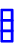}}
\DeclareRobustCommand{\T}{\includegraphics{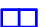}}
\DeclareRobustCommand{\TIs}{\includegraphics{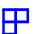}}
\DeclareRobustCommand{\TI}{\includegraphics{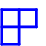}}
\DeclareRobustCommand{\TTs}{\includegraphics{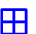}}
\DeclareRobustCommand{\TT}{\includegraphics{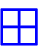}}
\DeclareRobustCommand{\TTT}{\includegraphics{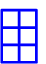}}
\DeclareRobustCommand{\Th}{\includegraphics{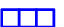}}
\DeclareRobustCommand{\ThI}{\includegraphics{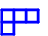}}
\DeclareRobustCommand{\ThII}{\includegraphics{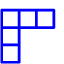}}
\DeclareRobustCommand{\ThTI}{\includegraphics{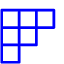}}
\DeclareRobustCommand{\ThThs}{\includegraphics{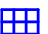}}
\DeclareRobustCommand{\ThTh}{\includegraphics{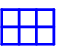}}
\DeclareRobustCommand{\ThThT}{\includegraphics{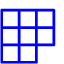}}
\DeclareRobustCommand{\ThThTs}{\includegraphics{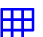}}
\DeclareRobustCommand{\ThThTh}{\includegraphics{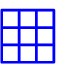}}
\DeclareRobustCommand{\ThThThs}{\includegraphics{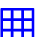}}
\DeclareRobustCommand{\F}{\includegraphics{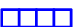}}
\DeclareRobustCommand{\FThThI}{\includegraphics{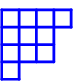}}
\DeclareRobustCommand{\Fi}{\includegraphics{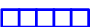}}
\DeclareRobustCommand{\FiFT}{\includegraphics{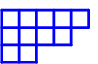}}

\DeclareRobustCommand{\Bx}{\includegraphics{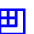}}
\DeclareRobustCommand{\Bxs}{\includegraphics{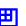}}

\newcommand{\Skew}[9]{\begin{picture}(51,31)(-3,-2.5)
     \put(-3,-2.5){\includegraphics{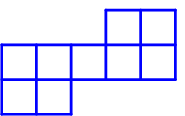}}
                              \put(30,20){\small$#1$}\put(40,20){\small$#2$}
                              \put(30,10){\small$#6$}\put(40,10){\small$#7$}
     \put( 0,10){\small$#3$}\put(10,10){\small$#4$}\put(20,10){\small$#5$}
     \put( 0, 0){\small$#8$}\put(10, 0){\small$#9$}
   \end{picture}}

\title[Experimentation in the Schubert Calculus]{Experimentation in the Schubert Calculus} 
\author[Mart\'in del Campo]{Abraham Mart\'in del Campo}
\address{Abraham Mart\'in del Campo, %\\
         IST Austria, %\\
         Am Campus 1, %\\
         3400 Klosterneuburg, %\\
         Austria
         }
\email{abraham.mc@ist.ac.at}
\author[Sottile]{Frank Sottile} 
\address{Frank Sottile, % \\
         Department of Mathematics, %\\
         Texas A\&M University, %\\
         College Station, %\\
         Texas \ 77843, %\\
         USA}
\email{sottile@math.tamu.edu}
\thanks{Based on Sottile's lectures at the Mathematical Society of
  Japan's 2012 Summer Institute on Schubert Calculus 16--20 July 2012 at Osaka City
  University.} 
\thanks{Research of Sottile supported in part by NSF grant DMS-1001615}
\thanks{This material is based upon work supported by the National Science 
Foundation under Grant No. 0932078 000, while Sottile was in 
residence at the Mathematical Science Research Institute (MSRI) in 
Berkeley, California, during the winter semester of 2013.}
\begin{document} 
 
\begin{abstract} 
 Many aspects of Schubert calculus are easily modeled on
 a computer.
 This enables large-scale experimentation to
 investigate subtle and ill-understood phenomena in the Schubert calculus.
 A well-known web of conjectures and results in the real Schubert calculus has been
 inspired by this continuing experimentation.
 A similarly rich story concerning intrinsic structure, or Galois groups, of Schubert
 problems is also beginning to emerge from experimentation.
 This showcases new possibilities for the use of computers in mathematical
 research.
\end{abstract} 
 
\maketitle 
 
%%%%%%%%%%%%%%%%%%%%%%%%%%%%%%%%% 
\section{Introduction} 
 The Schubert calculus of enumerative geometry is a rich and well-understood class of
 enumerative-geometric problems that are readily modeled on a computer.
 It provides a laboratory in which to investigate poorly understood phenomena in
 enumerative geometry using supercomputers.
 Modern software tools and available computer resources allow us to test
 billions of instances of Schubert problems for the phenomena we wish to study.
 This is  easily parallelized and
 therefore takes advantage of the current trend in computer architecture to increase
 computation power with increased parallelism.
 These computations have led to conjectures and new results and
 showcase new possibilities for the use of computers as a tool in mathematical
 research. 
  
 Of the solutions to a system of real polynomials or a geometric problem with real
 constraints, some may be real while the rest occur in complex conjugate pairs, and it is
 challenging to say anything meaningful about the distribution between the two types.
 Khovanskii showed that systems of polynomials with few monomials have upper bounds on
 their numbers of real solutions often far less than their numbers of complex
 solutions~\cite{Kh91}. 
 In contrast, geometric problems coming from the Schubert
 calculus on Grassmannians may have all their solutions be real~\cite{So97a,So99,Va06b},
 and while it is known only in some additional cases~\cite{So00c,Purbhoo}, this reality is  
 believed to hold for all flag manifolds.

 The real-number phenomena that we discuss is of a different character than these
 results. 
 The best-known involves the Shapiro Conjecture and its generalizations, where for certain
 classes of Schubert problems and natural choices of conditions, every
 solution is real. 
 Less understood are Schubert problems whose numbers of real solutions possess further 
 structure including lower bounds,  congruences, and gaps.

 Like field extensions, geometric problems have intrinsic structure encoded by Galois
 groups~\cite{J1870}.
 Unlike field extensions, little is known about such Galois groups.
 Work of Vakil~\cite{Va06b}, Billey and Vakil~\cite{BV}, and Leykin and
 Sottile~\cite{LS09} gives several avenues for studying Galois groups of Schubert problems
 on computers.
 Preliminary results suggest phenomena to study in 
 future large-scale computational experiments.
 For example, most Schubert problems have highly transitive Galois groups that contain the
 alternating group, while the rest are only singly transitive, 
 and the intrinsic structure restricting their Galois groups also 
 restricts their possible numbers of real solutions.

%%%%%%%%%%%%%%%%%%%%%%%%%%%%%%%%%%%%%%%%%%%%%%%%%%%%%%%%%%%%%%%%%%%%%%%%%%%%%%%%%
\begin{example}\label{Ex:four_lines}
 The classical problem of four lines asks: 
``how many lines in $\P^3$ meet four general lines?''
Three mutually skew lines 
\Blue{$\ell_1$}, \Red{$\ell_2$}, and \Green{$\ell_3$} 
lie on a unique hyperboloid. as shown in Fig.~\ref{F:4lines}.
%%%%%%%%%%%%%%%%%%%%%%%%%%%%%%%%%%%%%%%%%%%%%%%%%%%%%%%%%%%%%%%%%%%%%%%%%%%%%%%%%
\begin{figure}[htb]
 \begin{picture}(314,192)
   \put(3,0){\includegraphics[height=6.6cm]{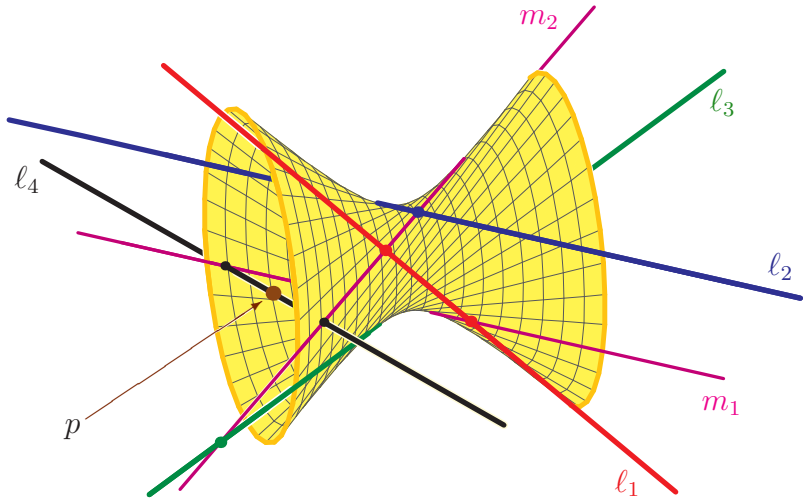}}
   \put(230,  3){\Red{$\ell_1$}}
   \put(288, 85){\Blue{$\ell_2$}}
   \put(266,147){\Green{$\ell_3$}}
   \put(  3,118){$\ell_4$}
   \put(263, 34){\Magenta{$m_1$}}
   \put(194,180){\Magenta{$m_2$}}

   \thicklines
   \put(30,30.5){\White{\line(3,2){63}}}
   \thinlines
   \put( 22, 25){$p$}\put(30,30.5){\Brown{\vector(3,2){66.6}}}
  \end{picture}

 \caption{Problem of four lines}\label{F:4lines}
\end{figure}
%%%%%%%%%%%%%%%%%%%%%%%%%%%%%%%%%%%%%%%%%%%%%%%%%%%%%%%%%%%%%%%%%%%%%%%%%%%%%%%%%
This hyperboloid has two rulings, one contains 
\Blue{$\ell_1$}, \Red{$\ell_2$}, and \Green{$\ell_3$}, 
and the second consists of the lines meeting these three.
If the fourth line, $\ell_4$, is general, then it will meet the hyperboloid in two points,
and through each of these points there is a unique line in the second ruling.
These two lines, \Magenta{$m_1$} and \Magenta{$m_2$}, are the solutions to this instance
of the problem of four lines.

If the four lines are real, then $\ell_4$ either meets the hyperboloid in two real points
(as in Fig.~\ref{F:4lines}) giving two real solution lines, or in 
two complex conjugate points giving two complex conjugate solutions.

The Galois/monodromy group of this problem is the group of permutations of the solutions
which arise by following the solutions over loops in the space of four-tuples 
$(\Blue{\ell_1}, \Red{\ell_2}, \Green{\ell_3}, \ell_4)$ of lines.
A simple such loop is described by rotating $\ell_4$ $180^\circ$ about the point 
\Brown{$p$}.
Following the two solutions along the loop interchanges them and shows that the
Galois/monodromy group of this problem is the full symmetric group $S_2$.

The Shapiro Conjecture asserts that if the lines 
$\Blue{\ell_1},\dotsc,\ell_4$ are tangent to a twisted
cubic at real points, then the two solutions are real.
Indeed, any three points on any twisted cubic are conjugate to any three points on
another, so it suffices to consider the cubic curve given by 
\[
   \Brown{\gamma}\ \colon\ t\ \longmapsto\ (12t^2-2\,,\, 7t^3+3t\,,\, 3t-t^3)\,,
\]
and the first three lines to be \Blue{$\ell(1)$}, \Red{$\ell(0)$}, and 
\Green{$\ell(-1)$}, where $\ell(t)$ is the line tangent to \Brown{$\gamma$} at the point
$\gamma(t)$. 
%%%%%%%%%%%%%%%%%%%%%%%%%%%%%%%%%%%%%%%%%%%%%%%%%%%%%%%%%%%%%%%%%%%%%%%%%%%%%%%%%%%
\begin{figure}[htb]
\centerline{
 \begin{picture}(318,180)(33,2)
   \put(33,0){\includegraphics[height=180pt]{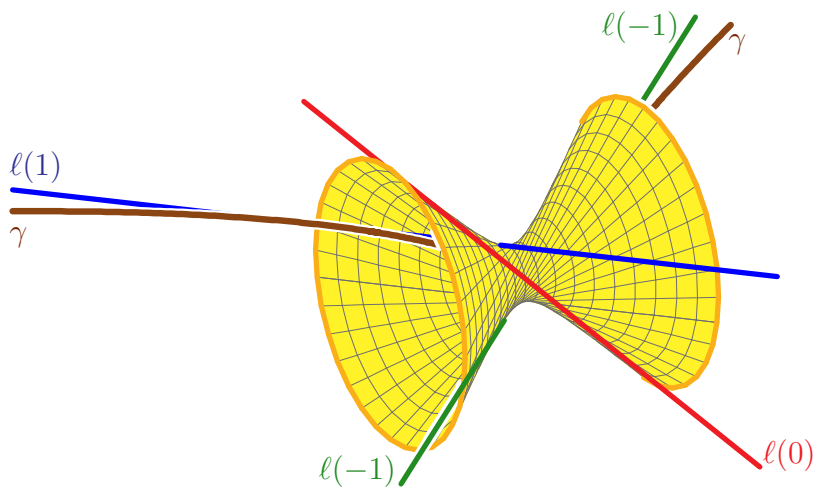}}
   \put(153,4){\Green{$\ell(-1)$}} \put(321,10){\Red{$\ell(0)$}}
   \put(36,119){\Blue{$\ell(1)$}} \put(36,94){\Brown{$\gamma$}}
   \put(260,172){\Green{$\ell(-1)$}} \put(308,167){\Brown{$\gamma$}}
  \end{picture}}
 \caption{Hyperboloid containing three lines tangent to \Brown{$\gamma$}.}
  \label{F1:hyperboloid} 
\end{figure}
%%%%%%%%%%%%%%%%%%%%%%%%%%%%%%%%%%%%%%%%%%%%%%%%%%%%%%%%%%%%%%%%%%%%%%%%%%%%%%%%%%%
As before, there is a unique hyperboloid (shown in Fig.~\ref{F1:hyperboloid}) that is 
ruled by the lines 
meeting all three, and the solutions to the problem of four lines 
correspond to points
where the fourth tangent line meets the hyperboloid.
%
%  Do not cut this text too much!
%

Consider the fourth line, $\ell(s)$, where $0<s<1$ (which we may assume as the three
intervals between \Blue{$\gamma(1)$}, \Red{$\gamma(0)$},
and \Green{$\gamma(-1)$} are projectively equivalent).
In Fig.~\ref{F1:line4}, we look down the throat of the hyperboloid
at the interesting part of this configuration.
As $\ell(s)$ is tangent to the branch of \Brown{$\gamma$} between \Blue{$\gamma(1)$}
and \Red{$\gamma(0)$}, it must meet the hyperboloid
in two real points. 
%%%%%%%%%%%%%%%%%%%%%%%%%%%%%%%%%%%%%%%%%%%%%%%%%%%%%%%%%%%%%%%%%%%%%%%%%%%%%%
 \begin{figure}[htb]
\centerline{
  \begin{picture}(287, 128)(-30,-14)
   \put(0,0){\includegraphics[height=100pt]{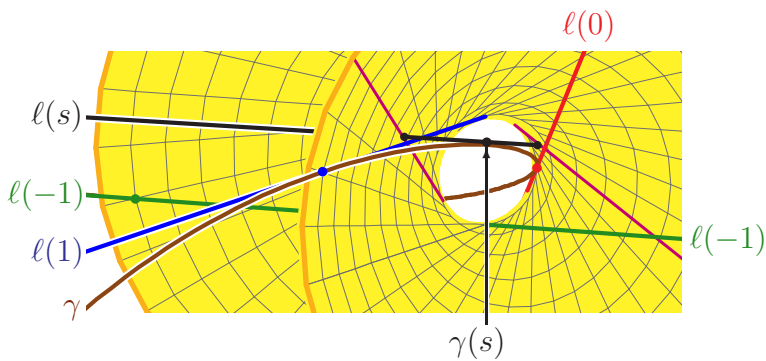}}
   \put(-21,20){\Blue{$\ell(1)$}} 
   \put(-30,42){\Green{$\ell(-1)$}} \put(228,25){\Green{$\ell(-1)$}} 
   \put(180,106){\Red{$\ell(0)$}}
   \put(-9,0){\Brown{$\gamma$}} 
   \put(-21,72){$\ell(s)$}
    \put(137,-14){$\gamma(s)$}
  \thicklines
  \thinlines
    \put(151.3,-4){\White{\vector(0,1){67}}} 
    \put(151.2,-4){\vector(0,1){67}} 
  \end{picture}}
 \caption{$\ell(s)$ meets the hyperboloid in two real points.} 
 \label{F1:line4}
 \end{figure} 
%%%%%%%%%%%%%%%%%%%%%%%%%%%%%%%%%%%%%%%%%%%%%%%%%%%%%%%%%%%%%%%%%%%%%%%%%
Through each point, there is a real line in the second ruling which
meets all four tangent lines, and this proves the Shapiro Conjecture
for this problem of four lines.
\hfill\qed
\end{example}
%%%%%%%%%%%%%%%%%%%%%%%%%%%%%%%%%%%%%%%%%%%%%%%%%%%%%%%%%%%%%%%%%%%%%%%%%%%%%%%%%

This paper is organized as follows.
In Section~\ref{S:background} we give background on the Shapiro Conjecture and the
Schubert calculus, and then explain how we may study Schubert problems on a computer.
In Section~\ref{S:generalizations} we discuss the Shapiro Conjecture more extensively, 
describing its generalizations and evidence that has been found for these
generalizations.
In Section~\ref{S:real_structure} we discuss additional structure that has been observed and
proven concerning the number of real solutions to the osculating Schubert calculus on
Grassmannians.
Finally, we close in Section~\ref{S:Galois} discussing several approaches to obtaining
information about Galois groups of Schubert problems, and sketch how they were used to
nearly determine the Galois groups of all Schubert problems on $\Gr(4,9)$.
\newpage

%%%%%%%%%%%%%%%%%%%%%%%%%%%%%%%%%%%%%%%%%%%%%%%%%%%%%%%%%%%%%%%%%%%%%%%%%%%%%%%%%
\section{Background}\label{S:background}

%%%%%%%%%%%%%%%%%%%%%%%%%%%%%%%%%%%%%%%%%%%%%%%%%%%%%%%%%%%%%%%%%%%%%%%%%%%%%%%%%
\subsection{The Shapiro Conjecture}\label{S:Shapiro}
 Given  univariate polynomials $f_1(t),\dotsc,f_k(t)$ of degree at most
 $n{-}1$, their \demph{Wronskian} of is the determinant of the matrix of their derivatives,
\[
   \defcolor{\Wr(f_1,\dotsc,f_k)}\ =\ 
    \det \left(  f_j^{(i-1)}(t) \mid i,j=1,\dotsc,k \right)\,.
\]
 Up to a scalar, this depends only on the linear span of the polynomials
 $f_1,\dotsc,f_k$. 
 Putting $f_1,\dotsc,f_k$ into echelon form with respect to the basis of monomials shows
 that we may assume $\deg f_1>\dotsb>\deg f_k$, from which we see that their Wronskian has
 degree at most $k(n{-}k)$.
 The Wronskian gives a map
 \begin{equation}\label{Eq:Wronski_map}
   \defcolor{\Wr}\ \colon\ \Gr(k,\C_{n-1}[t])\ \longrightarrow\ \P(\C_{k(n-k)}[t])\,,
 \end{equation}
 where \defcolor{$\C_d[t]$} is the space of polynomials in $t$ of degree at
 most $d$, $\Gr(k,\C_{n-1}[t])$ is the Grassmannian of $k$-dimensional linear subspaces of 
 $\C_{n-1}[t]$, and $\P(\C_{k(n-k)}[t])$ is the projective space of 1-dimensional linear
 subspaces of $\C_{k(n-k)}[t]$.
 The map~\eqref{Eq:Wronski_map} is surjective of degree 
 \begin{equation}\label{Eq:Wr-Deg}
%   \defcolor{\#_{k,n}}\ :=\ 
    \frac{ [k(n{-}k)]!\, 1! 2! \dotsb (k{-}1)!}%
         {(n{-}1)! (n{-}2)! \dotsb (n{-}k)!}\,,
 \end{equation}
 and each fiber consists of this number of points, counted with multiplicity~\cite{EH83}.
 The \demph{inverse Wronski problem} asks for the $k$-planes of polynomials with a given
 Wronskian. 
 This naturally arises in the theory of linear series on
 $\P^1$~\cite{EH83}, static output feedback control of linear systems~\cite{Byrnes}, and 
 mathematical physics~\cite{MV05}.

 The \demph{Shapiro Conjecture} posited that if 
 $\Psi(t)\in\P(\R_{k(n-k)}[t])$ had only real roots, then every $k$-plane of polynomials on
 $\Wr^{-1}(\Psi)$ is real.
 This was first studied computationally~\cite{So00b,Ve00}, then it was shown to be true if
 the roots of $\Psi$ were sufficiently clustered together~\cite{So99}.
 When $\min\{k,n{-}k\}=2$, it is equivalent to the statement that a rational function with
 only real critical points is essentially a quotient of real polynomials~\cite{EG02,EG11}.
 Finally, when a connection to integrable systems was realized, Mukhin, Tarasov, and
 Varchenko exploited that to prove the conjecture in full generality, eventually using
 that a symmetric matrix has only real eigenvalues~\cite{MTV_Sh,MTV_R}.
 See~\cite{FRSC} for a full account.

%%%%%%%%%%%%%%%%%%%%%%%%%%%%%%%%%%%%%%%%%%%%%%%%%%%%%%%%%%%%%%%%%%%%%%%%%%%%%%%%%
\subsection{Schubert calculus of enumerative geometry}\label{SS:SC}

The Schubert calculus of enumerative geometry consists of all problems 
of determining the linear subspaces that have 
specified positions with respect to other fixed, but general linear spaces. 
We broadly interpret it as the class of geometric problems which may be formulated as
intersecting sufficiently general Schubert varieties in flag manifolds.
%
%  sufficiently here is absolutely necessary.  Do not alter it lightly !
%
We will only describe the Schubert calculus on the Grassmannian in full.

The \demph{Grassmannian $\Gr(k,n)$} is the set of all $k$-dimensional linear subspaces 
of $\C^n$, which is an algebraic manifold of dimension $k(n{-}k)$.
%
%  Flag starts with `F'.
%
A \demph{flag $\Fdot$} is a sequence of linear subspaces
$
\Fdot : F_1 \subset F_2 \subset \dotsb \subset F_n,
$
where $\dim F_i = i$. 
A \demph{partition} 
$
\defcolor{\lambda} : (n-k) \geq \lambda_1 \geq \dotsb \geq \lambda_k \geq 0
$
is a weakly decreasing sequence of integers.  
A fixed flag $\Fdot$ and a partition $\lambda$ determine a \demph{Schubert variety}
\[
\defcolor{X_\lambda \Fdot} \, := \, 
\{ H\in \Gr(k,n) \ | \ 
\dim H\cap F_{n-k+i-\lambda_i}\geq i, \mbox{ for } i=1,\dotsc, k \} \, ,
\]
which has codimension $\defcolor{|\lambda|} := \lambda_1 + \dotsb + \lambda_k$
in $\Gr(k,n)$.
We often denote a partition $\lambda$ 
by its Young diagram, thus $\I$ denotes the partition $(1,0,\dotsc, 0)$.
As $\dim H\cap F_{n-k+i}\geq i$ for all $i$ and any $k$-plane $H$ and flag $\Fdot$, the
Schubert variety $X_{\Is}\Fdot$ consists of those $H$ with $H\cap F_{n-k}\supsetneq\{0\}$.
As  $|\I| = 1$, this is a hypersurface Schubert variety.

A \demph{Schubert problem} is a list of partitions
$\defcolor{\blambda}=(\lambda^1,\dotsc,\lambda^r)$ such that
$|\lambda^1| + \dotsb + |\lambda^r| = k(n-k)$. 
Given general flags $\Fdot^1,\dotsc, \Fdot^r$, Kleiman's Transversality
Theorem~\cite{Kl74} asserts that the intersection 
 \begin{equation}\label{Eq:KL74}
   X_{\lambda^1}\Fdot^1 \,\cap\, X_{\lambda^2}\Fdot^2 
    \,\cap\; \dotsb \;\cap\, X_{\lambda^r} \Fdot^r
 \end{equation}
is transverse and therefore zero-dimensional as 
$|\lambda^1| + \dotsb + |\lambda^r| = k(n-k)$.
The number, $\defcolor{d(\blambda)}$, of points in~\eqref{Eq:KL74} is independent of 
the choice of general flags. 
A zero-dimensional intersection is an \demph{instance} of the
Schubert problem $\blambda$. 
The points in~\eqref{Eq:KL74} are the \demph{solutions} to that instance.
We may write Schubert problems multiplicatively. 
For example, we write 
$\I \cdot \I \cdot \I \cdot \I = \I^4$
for the Schubert problem $(\I,\I,\I,\I)$ in $\Gr(2,4)$---this is the
problem of four lines in $\P^3$.
Then the number~\eqref{Eq:Wr-Deg} is $d(\I^{k(n-k)})$.

Any rational normal curve is projectively equivalent to
the curve 
\[
  \defcolor{\gamma(t)}\ :=\ (1,t,t^2/2,t^3/3!,\dotsc, t^{n-1}/(n{-}1)!)\,.
\]
For $t\in\C$ and any $1\leq i\leq n$, the $i$-plane osculating the curve $\gamma$ at
$\gamma(t)$ is
\[
   \defcolor{F_i(t)}\ :=\ \Span\{ \gamma(t), \gamma'(t),\dotsc, \gamma^{(i-1)}(t)\}\,.
\]
The flag \demph{$\Fdot(t)$ osculating} $\gamma$ at $\gamma(t)$ is the flag whose
subspaces are the $F_i(t)$.
The limit of $F_i(t)$ as $t\to\infty$ is the linear span of the last $i$ standard basis
vectors, and these subspaces form the flag \demph{$\Fdot(\infty)$}.
An instance of a Schubert problem $\blambda$ given by flags osculating $\gamma$ 
is an \demph{osculating instance} of $\blambda$.
Osculating flags are not general in the sense of Kleiman's
Theorem~\cite{Kl74}, as shown in~\cite[\S~2.3.6]{RSSS}.

The osculating Schubert calculus naturally arises in the study of linear series on $\P^1$,
where ramification at a point $t$ corresponds to membership in a Schubert variety
$X_{\lambda}\Fdot(t)$. 
An elementary consequence is the following useful proposition which provides a substitute
for Kleiman's Theorem for osculating flags.

%%%%%%%%%%%%%%%%%%%%%%%%%%%%%%%%%%%%%%%%%%%%%%%%%%%%%%%%%%%%%%%%%%%%%%%%%%%%%%%%%
\begin{proposition}\label{P:useful}
 Let $\lambda^1,\dotsc,\lambda^r$ be partitions and $t_1,\dotsc,t_r$ be distinct points of
 $\P^1$. 
 Then the intersection
 \begin{equation}\label{Eq:ROI}
   X_{\lambda^1}\Fdot(t_1) \,\cap\,
   X_{\lambda^2}\Fdot(t_2) \,\cap\; \dotsb  \;\cap\,
   X_{\lambda^r}\Fdot(t_r) 
 \end{equation}
 has dimension $k(n{-}k)-|\lambda^1|-\dotsb-|\lambda^r|$, so that if 
 $\lambda^1,\dotsc,\lambda^r$ is a Schubert problem, it is zero-dimensional.
 Furthermore, if $H\in\Gr(k,n)$, then there is a unique Schubert problem
 $\lambda^1,\dotsc,\lambda^r$ and unique points $t_1,\dotsc,t_r\in\P^1$ such that $H$ lies
 in the intersection~\eqref{Eq:ROI}.
\end{proposition}
%%%%%%%%%%%%%%%%%%%%%%%%%%%%%%%%%%%%%%%%%%%%%%%%%%%%%%%%%%%%%%%%%%%%%%%%%%%%%%%%%

Schubert problems are 
efficiently represented on a computer through local coordinates and determinantal
equations. 
The set of all $k$-planes $H\in\Gr(k,n)$ not in $X_{\Is}\Fdot(\infty)$ 
is identified with the space of $k\times(n{-}k)$ matrices $X$ via 
$X\mapsto \mbox{row space}[I_k:X]$, where $I_k$ is the identity matrix.
This forms a dense open subset of the Grassmannian and 
the entries of $X$ give local coordinates for $\Gr(k,n)$.

We formulate membership in Schubert varieties (and thus Schubert problems) in terms of
determinantal equations.
If we represent an $i$-dimensional subspace $F_i$ as the  
row space of a full rank $i\times n$ matrix (also denoted by $F_i$) and $H$ by a
$k\times n$ matrix, then
\begin{equation}\label{eq:incidence_condition}
\dim H\cap F_i \geq j \iff \rank 
\left[\begin{array}{c} 
H  \\F_i 
\end{array}\right] 
\leq k+i-j \ ,
\end{equation}
which is defined by the vanishing of all $(k{+}i{-}j{+}1)\times (k{+}i{-}j{+}1)$ minors. 
Therefore, when representing a flag $\Fdot$ by a $n\times n$ matrix whose first
$i$ rows span $F_i$,~\eqref{eq:incidence_condition} (with $i$ replaced by
$n{-}k{+}j{-}\lambda_j$) gives polynomial equations for the
Schubert variety $X_\lambda \Fdot$ in the affine patch $[I_k:X]$.
When desired, we may use similar smaller coordinate patches for $X_\lambda\Fdot(\infty)$
and $X_\lambda\Fdot(\infty)\cap X_\mu\Fdot(0)$.

%%%%%%%%%%%%%%%%%%%%%%%%%%%%%%%%%%%%%%%%%%%%%%%%%%%%%%%%%%%%%%%%%% 
\subsection{Experimentation on a supercomputer}

Equations for Schubert problems based on~\eqref{eq:incidence_condition} may be solved
in some sense using software tools, and information extracted that is
relevant to the questions we are studying (e.g.\ real solutions or Galois groups), again
using software tools.
This ability to study individual instances of Schubert problems on a computer becomes a
powerful method of investigation when automated, for literally billions of instances of
thousands to millions of Schubert problems may be studied.

The challenge posed by scaling computations from the few score to the billions is
two-fold---it requires careful organization {\it and} access to computational resources.
The fundamental observation which allows this scale of investigation is that it is
intrinsically parallel.
Computing/studying one instance of a Schubert problem is independent of any other instance.
This enables us to take advantage of current widely available computer
resources---multiprocessor computational servers, established computer clusters,
as well as ad hoc resources for our investigations.
For example, most of the experimentation in~\cite{FRSC_secant,FRSC_monotone} was done on
the Calclabs, which consists of over 200 Linux workstations that moonlight as a Beowulf
cluster---their day job being calculus instruction, and that in~\cite{Lower_Exp} used the
brazos cluster at Texas A\&M University in which our research
group controls 20 eight-core nodes.

The challenge of organizing a computational investigation on this scale, as
well as ensuring that it is repeatable and robust, is met through modern
software tools.
These include organizing the computation with a database, monitoring it with web-based
tools, running the computation using a job scheduler, as well as the core code itself,
written in a scripting language to communicate with the database and organize the parts
of the computation which are carried out by special purpose optimized software that is
either widely available or written by our team.

The structure of these investigations is due to Chris Hillar.
A detailed description of the experimental design and its execution is 
in~\cite{Exp-FRSC}, which explains our paradigm for large-scale
experimentation using supercomputers and modern software tools.
We give few details here, more may be found in the individual papers referenced.

%%%%%%%%%%%%%%%%%%%%%%%%%%%%%%%%%%%%%%%%%%%%%%%%%%%%%%%%%%%%%%%%%%
\section{History and generalizations of the Shapiro 
Conjecture}\label{S:generalizations} 

The Shapiro Conjecture of Subsection~\ref{S:Shapiro} may alternatively be formulated in
terms of the osculating Schubert calculus.\medskip

%%%%%%%%%%%%%%%%%%%%%%%%%%%%%%%%%%%%%%%%%%%%%%%%%%%%%%%%%%%%%%%%%%%%%%%%%%%%%%%%%
\noindent{\bf Shapiro Conjecture }(Theorem of Mukhin, Tarasov, and
 Varchenko~\cite{MTV_Sh,MTV_R}){\bf .}
{\it
  Let $\blambda = (\lambda^1,\dotsc, \lambda^r)$ be a Schubert problem in $\Gr(k,n)$ and
  let $t_1,\dotsc, t_r$ be distinct real numbers. 
  The intersection  
 \begin{equation}\label{Eq:OscInst}
   X_{\lambda^1}\Fdot(t_1)\,\cap\,
   X_{\lambda^2}\Fdot(t_2)\,\cap\;\dotsb\;\cap\,
   X_{\lambda^r}\Fdot(t_r) 
 \end{equation}
 is transverse and consists of $d(\blambda)$ real points.}\medskip
%%%%%%%%%%%%%%%%%%%%%%%%%%%%%%%%%%%%%%%%%%%%%%%%%%%%%%%%%%%%%%%%%%%%%%%%%%%%%%%%%

The connection between this and the Wronskian formulation of Subsection~\ref{S:Shapiro} is
given carefully in~\cite{FRSC} and~\cite[Ch.~10]{IHP}.
The main idea is straightforward.
A univariate polynomial $f(t)$ of degree at most $n{-}1$ is a linear
form $\Lambda$ evaluated on the rational normal curve $\gamma(t)$. 
A linearly independent set $f_1,\dotsc, f_k$ of univariate polynomials of
degree $n{-}1$ gives independent linear forms $\Lambda_1, \dotsc, \Lambda_k$.
Thus $H := \ker(\Lambda_1,\ldots, \Lambda_k)$ lies in $\Gr(n{-}k,n)$. 
The following is a calculation.

%%%%%%%%%%%%%%%%%%%%%%%%%%%%%%%%%%%%%%%%%%%%%%%%%%%%%%%%%%%%%%%%%%%%%%%%%%%%%%%%%
\begin{lemma}
 Let $f_1,\dotsc, f_k$ be polynomials in $\C_{n-1}[t]$ coming from linear forms
 $\Lambda_1, \dotsc, \Lambda_k$ and set $H := \ker(\Lambda_1,\ldots, \Lambda_k)$. 
 Then $t$ is a root of the Wronskian $\Wr(f_1,\dotsc, f_k)$ if and only if 
 $H\in X_{\Is}\Fdot(t)$.
\end{lemma}
%%%%%%%%%%%%%%%%%%%%%%%%%%%%%%%%%%%%%%%%%%%%%%%%%%%%%%%%%%%%%%%%%%%%%%%%%%%%%%%%%

%%%%%%%%%%%%%%%%%%%%%%%%%%%%%%%%%%%%%%%%%%%%%%%%%%%%%%%%%%%%%%%%%%%%%%%%%%%%%%%%%
\subsection{The Shapiro Conjecture for other flag manifolds}\label{S:other}

Let $\eta$ be a regular nilpotent element of the Lie algebra of 
a group $G$ (the closure of its adjoint orbit contains all nilpotents).
Then $t\mapsto \mbox{exp}(t\eta)\in G$ is a subgroup $\Gamma(t)$ of $G$.
When $G=\mbox{SL}_n\C$ and $\eta$ has all entries zero except for  a 1 in each 
position $(i,i{+}1)$, the matrix $\Gamma(t)$ represents the
flag $\Fdot(t)$ with $\Gamma(t).\Fdot(0)=\Fdot(t)$.
Then~\eqref{Eq:OscInst} becomes
\[
   \Gamma(t_1).X_{\lambda^1}\Fdot(0) \,\cap\, \Gamma(t_2).X_{\lambda^2}\Fdot(0) 
    \,\cap\; \dotsb \;\cap\, \Gamma(t_r).X_{\lambda^r}\Fdot(0)\,.
\]
This gives osculating instances of Schubert problems, 
and therefore a version of the Shapiro Conjecture, for any flag
manifold.

Purbhoo proved that the Shapiro Conjecture holds for the orthogonal
Grassmannian~\cite{Purbhoo}, but it is known to fail for other non-Grassmann\-ian flag
manifolds.  
For the Lagrangian Grassmannian and type $A$ flag manifolds, 
the conjecture may be repaired.
For the Lagrangian Grassmannian, see~\cite[Sec.~7.1]{FRSC}.
For flag manifolds of type $A$ a counterexample was found in~\cite{So00b}.
We present the simplest counterexample.

%%%%%%%%%%%%%%%%%%%%%%%%%%%%%%%%%%%%%%%%%%%%%%%%%%%%%%%%%%%%%%%%%%%%%%%%%%%%%%%%%
\begin{example}\label{Ex:monotone}
Let $\Fl(2,3;4)$ be the manifold of flags $\Magenta{m\subset M}$ in $\P^3$ where
\Magenta{$m$} is a line and \Magenta{$M$} is a 2-plane.
Consider the problem in $\Fl(2,3;4)$ where \Magenta{$m$} meets three lines 
\Blue{$\ell_1$}, \Red{$\ell_2$}, and \Green{$\ell_3$} and \Magenta{$M$} contains
two points $p,q$. 
Then \Magenta{$M$} contains the line $\overline{pq}$ and 
so \Magenta{$m$} meets $\overline{pq}$.
Thus \Magenta{$m$} is one of two solutions to the problem of four lines given by 
\Blue{$\ell_1$}, \Red{$\ell_2$}, \Green{$\ell_3$}, and $\overline{pq}$, and 
\Magenta{$M$} is the span of \Magenta{$m$} and $\overline{pq}$.

Consider osculating instances of this Schubert problem where the lines are
tangent to the rational normal curve \Brown{$\gamma$} of Example~\ref{Ex:four_lines} and the
points lie on \Brown{$\gamma$}.
Let \Blue{$\ell(1)$}, \Red{$\ell(0)$}, and 
\Green{$\ell(-1)$} be the tangent lines and $\gamma(s), \gamma(t)$ the points.
Consider the auxiliary problem of \Magenta{$m$} meeting these three tangent lines as well
as the secant line $\ell(s,t):=\overline{\gamma(s) \gamma(t)}$.

When $0<s<t<1$ as in Fig.~\ref{F4:monotone}, the line $\ell(s,t)$ meets
 %%%%%%%%%%%%%%%%%%%%%%%%%%%%%%%%%%%%%%%%%%%%%%%%%%%%%%%%%%%%%%%%%%%%%%%%%%%%%
 \begin{figure}[htb]
 \centerline{
  \begin{picture}(352,146)(-28,-13)
   \put(0,0){\includegraphics[height=120pt]{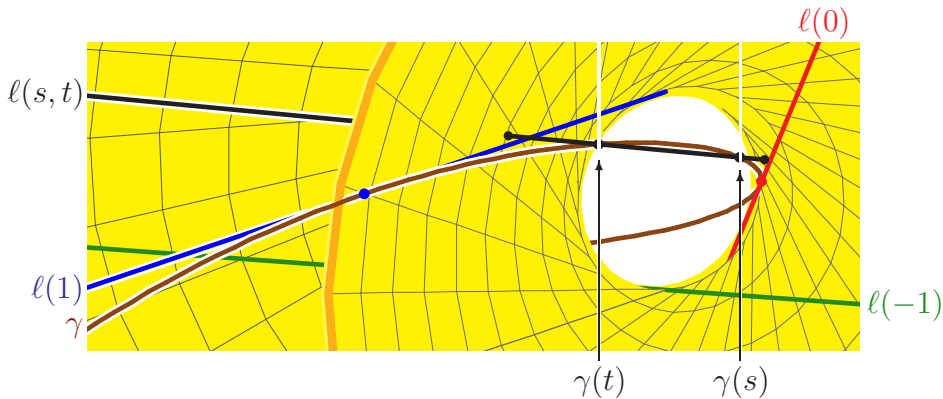}} 
   \put(-19,22){\Blue{$\ell(1)$}} \put(297,16){\Green{$\ell(-1)$}} 
   \put(271,123){\Red{$\ell(0)$}}
   \put(-6,9){\Brown{$\gamma$}} \put(-28,96){$\ell(s,t)$}
 \thicklines
     \put(195.8,-2){\White{\vector(0,1){152}}}
     \put(249.2,-2){\White{\vector(0,1){149}} }
  \thinlines
    \put(186.1,-13){$\gamma(t)$}  \put(195.8,-2){\vector(0,1){76}}
    \put(239,-13){$\gamma(s)$}   \put(249.2,-2){\vector(0,1){72}} 
  \end{picture}
}
\caption{A secant line meeting the hyperboloid.}
\label{F4:monotone}
 \end{figure} 
%%%%%%%%%%%%%%%%%%%%%%%%%%%%%%%%%%%%%%%%%%%%%%%%%%%%%%%%%%%%%%%%%%
the hyperboloid in two real points.
As before, there are two real lines \Magenta{$m$} and two real solutions 
\Magenta{$m\subset M$} to our Schubert problem. 
Observe that there is nothing special in the choice of $s,t$ of Fig.~\ref{F4:monotone},
for any choice of $0<s<t<1$ the line $\ell(s,t)$ meets the hyperboloid in two real points,
and so both solutions \Magenta{$m\subset M$} to our Schubert problem will be real.
 
In contrast, Fig.~\ref{F5:non-monotone} shows an example when $-1<s < 0< t < 1$ and
the secant line $\ell(s,t)$ does not meet the hyperboloid in two real points.
%%%%%%%%%%%%%%%%%%%%%%%%%%%%%%%%%%%%%%%%%%%%%%%%%%%%%%%%%%%%%%%%%%%%%%%%%%%%%%
 \begin{figure}[htb]
 \centerline{
  \begin{picture}(323,158)(-26,-11)
   \put(0,0){\includegraphics[height=135pt]{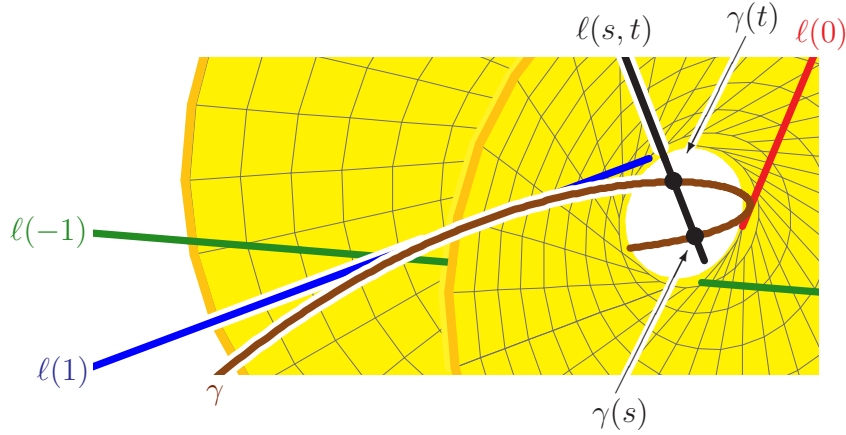} }
   \put(-15,5){\Blue{$\ell(1)$}} %\put(168,19){\Green{$\ell(-1)$}} 
   \put(-25,58){\Green{$\ell(-1)$}} 
   \put(272,133){\Red{$\ell(0)$}}
   \put(49,-2){\Brown{$\gamma$}}\put(189,134){$\ell(s,t)$}
   \put(246,139){$\gamma(t)$}
   \put(253.1,135){\vector(-1,-2){22}}
%       \put(150,120){\vector(-1,-2){22}}
   \put(195,-11){$\gamma(s)$}  
   \put(204.1,0){\vector(1,2){27}}
%       \put(121,0){\vector(1,2){27}} 
  \end{picture}
 }
 \caption{A secant line not meeting the hyperboloid.}
 \label{F5:non-monotone}
 \end{figure}
%%%%%%%%%%%%%%%%%%%%%%%%%%%%%%%%%%%%%%%%%%%%%%%%%%%%%%%%%%%%%%%%%%%%%%%%%
In this case, the two solutions \Magenta{$m$} to the auxiliary problem are complex
conjugates, and the same is true for the solutions \Magenta{$m\subset M$} to our Schubert
problem.
Thus the assertion of the Shapiro Conjecture does not hold for this Schubert problem.
\hfill\qed
\end{example}
%%%%%%%%%%%%%%%%%%%%%%%%%%%%%%%%%%%%%%%%%%%%%%%%%%%%%%%%%%%%%%%%%%%%%%%%%%%%%%%%%

This failure of the Shapiro Conjecture is particularly interesting. 
If we label the points $-1, 0, 1$ with \Blue{$1$} to indicate they are conditions on the
line \Magenta{$m$}, and the points $s, t$  with \Magenta{$2$} to indicate they are
conditions on the plane \Magenta{$M$}, then these labels  come in the following orders
along $\gamma$ 
\begin{equation}\label{eq:monotone sequence}
   \Blue{111}\Magenta{22} \mbox{ in Fig.~\ref{F4:monotone}} \quad \mbox{and }\quad 
   \Blue{11}\Magenta{2}\Blue{1}\Magenta{2} \mbox{ in Fig.~\ref{F5:non-monotone}}. 
\end{equation}
The first sequence is \demph{monotone} and both solutions are always real, while the second 
sequence is not monotone and the two solutions are not necessarily real.
This leads to a version of the Shapiro Conjecture for flag manifolds, and to
two further extensions.
We formalize these ideas.

Let $\alphdot : 0< a_1< \dotsb < a_p < n$ be a sequence of integers.
A \demph{flag $\Edot$ of type $\alphdot$}
is a sequence of linear subspaces
$\Edot \colon E_{a_1} \subset E_{a_2}\subset \dotsb \subset E_{a_p}$, 
were $\dim E_{a_i} = a_i$. 
The set of all such sequences is the \demph{flag manifold $\Fl(\alphdot; n)$}, which 
has dimension $\dim(\alphdot):= \sum_{i=1}^p (n{-}a_i)(a_i{-}a_{i-1})$, where $a_0=0$.
When $p=1$, this is the Grassmannian $\Gr(a_1,n)$.

Consider the projections
$\defcolor{\pi_{a_i}}\colon \Fl(\alphdot; n) \to \Gr(a_i, n)$
given by $\Edot \mapsto E_{a_i}$. 
A \demph{Grassmannian Schubert variety} has the form 
$\pi^{-1}_{b}X_{\lambda}\Fdot$, where $b\in \alphdot$ and
$\lambda$ is a partition for $\Gr(b,n)$.
We will write \defcolor{$X_{(\lambda,b)}\Fdot$} for $\pi^{-1}_{b}X_{\lambda}\Fdot$.
A list 
$\defcolor{(\blambda, \bb)}:=
    \bigl((\lambda^1,b_1), (\lambda^2,b_2), \dotsc , (\lambda^r,b_r)\bigr)$,
with $|\lambda^1| + \dotsb + |\lambda^r| = \dim \Fl(\alphdot;n)$
is a  \demph{Grassmannian Schubert problem}.
A list of real numbers $t_1, \dotsc, t_r 
\in \R$ is \demph{monotone} with respect to the Grassmannian Schubert problem 
$(\blambda, \bb)$, if  $t_i<t_j$ whenever $b_i < b_j$.
More generally, if $\prec$ is any cyclic order on $\R\P^1$, then
$t_1,\dotsc,t_r\in\R\P^1$ is monotone with respect to $(\blambda, \bb)$, if  
$b_i<b_j\Rightarrow t_i \prec t_j$. \medskip

\noindent{\bf Monotone Conjecture. }
{\it
  Let $(\blambda,\bb) = \bigl((\lambda^1,b_1),\dotsc, (\lambda^r,b_r)\bigr)$ be a
  Grassmannian Schubert problem in $\Fl(\alphdot;n)$. 
  If $t_1,\dotsc, t_r\in\R\P^1$ is monotone with respect to $(\blambda, \bb)$,
  then intersection  
 \[
   X_{(\lambda^1,b_1)}\Fdot(t_1)\,\cap\,
   X_{(\lambda^2, r_2)}\Fdot(t_2)\,\cap\;\dotsb\;\cap\,
   X_{(\lambda^r,b_r)}\Fdot(t_r) 
 \]
 is transverse with all points of intersection real.}\medskip

This conjecture was first noted in~\cite{So00c}. 
A formulation for two- and three-step flags was given in~\cite{So03}  together with
computational evidence supporting it. 
The general statement was made in~\cite{RSSS}, which reported on an experiment testing
1,140 Schubert problems on 29 flag manifolds, solving more than 525 million of random
instances and verifying the Monotone Conjecture in each of more than 158 million monotone
instances. 
These computations  took 15.76 gigaHertz-years.

We explain how the number of real solutions was determined. 
For a Grassmannian Schubert problem 
$\bigl((\lambda^1, b_1), \dotsc , (\lambda^r, b_r)\bigr)$, 
select $r$ random points on the rational normal curve $\gamma$ and construct osculating
flags. 
Using these flags,  represent the Schubert problem as a system of polynomial
equations given by the determinantal conditions in~\eqref{eq:incidence_condition}
in some system of local coordinates for $\Fl(\alphdot;n)$. 
Then eliminate all but one variable from the equations, obtaining an eliminant. 
When the eliminant is square-free and has degree equal to 
the expected number of complex solutions, the Shape Lemma guarantees that the number of real
solutions to the Schubert problem  equals the number of real roots of the eliminant, which
may be computed using Sturm sequences.

A given set of $r$ points is permuted in each of a predetermined set of orders along
$\R\P^1$ (called \demph{necklaces}) to give different orders along $\R\P^1$ in which the
conditions are evaluated, and for each the number of real solutions is determined.
The result is recorded in a frequency table for that Schubert problem which records how
often a given number of real solutions was observed for a given necklace.

To illustrate the data obtained in this experiment, consider the Schubert problem
$(\I,\Blue{2})^7 \cdot (\T,\Magenta{3})^2$ in $\Fl(2,3;6)$, which looks for the flags
\Magenta{$m\subset M$} in $\P^6$, where \Magenta{$m$} 
is 2-plane that meets seven 4-planes and \Magenta{$M$} is a three-plane that meets
two 2-planes. 
This problem has 14 solutions.
Table~\ref{T:W1^7X1^2=14} is the frequency table obtained from computing 800000 random
osculating instances of this problem. 
%%%%%%%%%%%%%%%%%%%%%%%%%%%%%%%%%%%%%%%%%%%%%%%%%%%%%%%%%%%%%%%%%%%%%%%%%%%%%%%%%%%%
\begin{table}[htb]
\caption{Frequency table for $(\I,\Blue{2})^7 \cdot (\T,\Magenta{3})^2 $ in $\Fl(2,3;6)$.}\vspace{-10pt}
\label{T:W1^7X1^2=14}
\noindent{

\noindent\begin{tabular}{|r||r|r|r|r|r|r|r|r||r|}
\multicolumn{10}{c}{Real Solutions}\\
\cline{1-10}
%\multirow{10}{*}
 & 0 & 2 & 4 & 6 & 8 & 10 & 12 & 14 &Total\\\cline{1-10}\noalign{\smallskip}\cline{1-10}
\Blue{2}\Blue{2}\Blue{2}\Blue{2}\Blue{2}\Blue{2}\Blue{2}\Magenta{3}\Magenta{3} 
&&&&&&&&200000 & 200000\\\cline{1-10}
\Blue{2}\Blue{2}\Blue{2}\Blue{2}\Blue{2}\Magenta{3}\Blue{2}\Blue{2}\Magenta{3} 
& & & 22150 & 8705 & 34833 & 45439 & 39481 & 49392 & 200000 \\\cline{1-10}
\Blue{2}\Blue{2}\Blue{2}\Blue{2}\Blue{2}\Blue{2}\Magenta{3}\Blue{2}\Magenta{3} 
& & & 24773 & 10591 & 14377 & 11029 & 8033 & 131197& 200000\\\cline{1-10}
\Blue{2}\Blue{2}\Blue{2}\Blue{2}\Magenta{3}\Blue{2}\Blue{2}\Blue{2}\Magenta{3} 
& 5& 52 & 3146 & 16758 & 42337 & 66967 & 50282 & 20453 & 200000\\\cline{1-10}\noalign{\smallskip}\cline{1-10}
Total & 5& 52 & 50069 & 36054 & 91547& 123435 & 97796 & 401042 & 800000 \\\cline{1-10}
\end{tabular}

}
\end{table}
%%%%%%%%%%%%%%%%%%%%%%%%%%%%%%%%%%%%%%%%%%%%%%%%%%%%%%%%%%%%%%%%%%%%%%%%%%%%%%%%%%%%
The columns are indexed by even integers from 0 to 14 for the possible numbers of real
solutions.
The rows are indexed by the possible necklaces, using the notation of~\eqref{eq:monotone
  sequence}. 
The first row labeled  with $\Blue{2222222}\Magenta{33}$ represents tests of the Monotone 
Conjecture, verifying it in 200000 instances as the only entries lie in the column for 14 
real solutions.

In addition to the computational evidence in favor of the Monotone Conjecture, 
theoretical evidence was provided by Eremenko, Gabrielov, Shapiro, and
Vainshtein~\cite{EGSV} who proved the conjecture for all Schubert problems in $\Fl(1,2;n)$
and $\Fl(n{-}2, n{-}1; n)$.  
Their result can be formulated in $\Gr(n{-}2,n)$, where it becomes a statement about real
points of intersection of Schubert varieties given by flags that are secant to a rational
normal curve $\gamma$ in a specific way. 
This condition on the secant flags makes sense for any Grassmannian, and leads to
a second generalization of the Shapiro Conjecture.

A flag $\Fdot$ is \demph{secant along an interval $I$} of a rational normal curve {$\gamma$}
if each subspace $F_i$ is spanned by its points of intersection with $I$.
Secant flags are \demph{disjoint} if the intervals of secancy are pairwise
disjoint. \medskip

\noindent{\bf Secant Conjecture. }
{\it
  Let $\blambda = (\lambda^1,\dotsc, \lambda^r)$ be a Schubert 
  problem in $\Gr(k,n)$.
  If $\Fdot^1,\dotsc, \Fdot^r$ are disjoint secant flags, then intersection  
 \[
   X_{\lambda^1}\Fdot^1\,\cap\,
   X_{\lambda^2}\Fdot^2\,\cap\; \dotsb\; \cap\,
   X_{\lambda^r}\Fdot^r 
 \]
 is transverse with all points of intersection real.}\medskip

The Secant Conjecture holds in two special cases beyond  the Grassmannian 
$\Gr(n{-}2,n)$ that was shown in~\cite{EGSV}.
A family of secant flags becomes osculating in the limit as the intervals of secancy
shrink to a point.  
In this way, the limit of the Secant Conjecture is the Shapiro Conjecture
(Theorem of Mukhin, Tarasov, and Varchenko~\cite{MTV_Sh,MTV_R})
and so the Secant Conjecture is true when the points of secancy are sufficiently
clustered.
The special case when the points of secancy form arithmetic sequences and 
the Schubert problem is $\I^{k(n-k)}$ was shown in~\cite{MTV_XXX}.

The strongest evidence for the Secant Conjecture is an experiment 
that used 1.07 teraHertz-years of computation, testing more than 498 millions of instances
of the Secant Conjecture in 703 Schubert problems on 13 Grassmannians.
This is reported in~\cite{FRSC_secant}.
As with the Monotone Conjecture, these computations relied upon counting the number of
real roots of an eliminant. 
Table~\ref{T:W31^4=9} displays the data obtained for the Schubert problem $\ThI^4$ in
$\Gr(4,8)$ with 9 solutions. 
%%%%%%%%%%%%%%%%%%%%%%%%%%%%%%%%%%%%%%%%%%%%%%%%%%%%%%%%%%%%%%%%%%%%%%%%%%%%%%%%%%%%
\begin{table}[htb]
\caption{Real solutions for $\ThI^4$ in $\Gr(4,8)$.}\vspace{-10pt}
\label{T:W31^4=9}
\noindent{

\noindent\begin{tabular}{r|r||r|r|r|r|r|r|r|c||r|}
\multicolumn{11}{c}{Overlap Number}\\
\cline{2-11}
\multirow{8}{*}{\begin{sideways}Real Solutions\end{sideways}} &
 & 0&1&2&3&4&5&6 &$\dotsb$ &Total\\\cline{2-11}\noalign{\smallskip}\cline{2-11}
& 1 &   &   &    &       &     &     &   16 &$\dotsb$ & 758 \\\cline{2-11}
& 3 &   &   &    &       &  7 & 612 & 783 &$\dotsb$ & 18276 \\\cline{2-11}
& 5 &   &   &    & 123& 659 & 4541 & 4847 &$\dotsb$ & 79173 \\\cline{2-11}
& 7 &   &   &    & 158& 663 & 3804 & 4545 &$\dotsb$ & 91536 \\\cline{2-11}
& 9 &141420& & 4051& 7937 & 11241& 17310 & 15705 &$\dotsb$ & 310257 \\\cline{2-11}
&Total & 141420& 0& 4051& 8218& 12570& 26267 & 25896 &$\dotsb$& 500000 \\\cline{2-11}
\end{tabular}

}
\end{table}
%%%%%%%%%%%%%%%%%%%%%%%%%%%%%%%%%%%%%%%%%%%%%%%%%%%%%%%%%%%%%%
Its rows are indexed by the odd numbers from 1 to 9 for the
possible number of real solutions. 
The columns are indexed by the \demph{overlap number}, which measures intersections 
between secant flags. 
The overlap number is 0 if and only if the flags are disjoint. 
Thus, the first column in Table~\ref{T:W31^4=9} represents tests of the Secant 
Conjecture, verifying it in the 141420 instances computed.

The column corresponding to overlap number one is empty as this cannot be attained by the 
intervals of secancy for this problem. 
Another interesting feature in Table~\ref{T:W31^4=9} 
is the column corresponding to overlap number six, which are flags that are very slightly
non-disjoint.
For this column, the solutions were also all real, while in the next column, at least five
were real. 
It is only with overlap number six and beyond that we found instances with only one real
solution. 

The Monotone Conjecture and Secant Conjecture have a common generalization, the Monotone-Secant 
Conjecture.
Disjoint flags have a naturally occurring order along $\R\P^1$.
A list $\Fdot^1, \dotsc, \Fdot^r$ of disjoint secant flags is \demph{monotone} with
respect to a Grassmannian Schubert problem $(\blambda, \bb)$, if $\Fdot^i$ proceeds
$\Fdot^j$ whenever $b_i < b_j$. \medskip

\noindent{\bf Monotone-Secant Conjecture. }
{\it
  Let $(\blambda,\bb) = 
    \bigl((\lambda^1,b_1),\dotsc, (\lambda^r,b_r)\bigr)$ be a Grassmannian Schubert 
  problem in $\Fl(\alphdot;n)$. If a list
  $\Fdot^1,\dotsc, \Fdot^r$ of disjoint secant flags is monotone with respect to $(\blambda, \bb)$,
  then intersection  
 \[
   X_{(\lambda^1,b_1)}\Fdot^1\,\cap\,
   X_{(\lambda^2,b_2)}\Fdot^2\,\cap\;\dotsb\;\cap\,
   X_{(\lambda^r,b_r)}\Fdot^r 
 \]
 is transverse with all points of intersection real.}\medskip

This conjecture is currently being studied on a supercomputer. 
As of 10 August 2013, we have tested 11,139,880,850 instances of 1300 Schubert problems
taking 1.894 teraHertz-years. 
Of these, 254,889,515 were instances of the Monotone-Secant Conjecture where
the conjecture was verified. 
Other computations (262,941,321) involved instances of the Monotone Conjecture for
comparison.  
Table~\ref{T2:W1^7X1^2=14} displays the data obtained for Monotone-Secant instances of the
Schubert problem $(\I,\Blue{2})^7 \cdot (\T,\Magenta{3})^2$ in $\Fl(2,3;6)$. 
%%%%%%%%%%%%%%%%%%%%%%%%%%%%%%%%%%%%%%%%%%%%%%%%%%%%%%%%%%%%%%%%%%%%%%%%%%%%%%%%%%%%
\begin{table}[htb]
\caption{Frequency table for $(\I,\Blue{2})^7 \cdot (\T,\Magenta{3})^2 $ in 
      $\Fl(2,3;6)$.}\vspace{-10pt}
\label{T2:W1^7X1^2=14}
\noindent{

\begin{tabular}{|r||r|r|r|r|r|r|r|r||r|}
\multicolumn{10}{c}{Real Solutions}\\\hline
 & {\!0\!} & {\!2\!} & 4 & 6 & 8 & 10 & 12 & 14 &Total\\\hline\hline
\Blue{2}\Blue{2}\Blue{2}\Blue{2}\Blue{2}\Blue{2}\Blue{2}\Magenta{3}\Magenta{3} 
&&&&&&&&400000 & 400000\\\hline
\Blue{2}\Blue{2}\Blue{2}\Blue{2}\Blue{2}\Magenta{3}\Blue{2}\Blue{2}\Magenta{3} 
& & & 131815 & 51761 & 92849 & 73988 & 27054 & 22533 & 400000\\\hline
\Blue{2}\Blue{2}\Blue{2}\Blue{2}\Blue{2}\Blue{2}\Magenta{3}\Blue{2}\Magenta{3} 
& &  & 142271 & 43847 & 36252 & 40595 & 22399 & 114636 & 400000\\
\hline
\Blue{2}\Blue{2}\Blue{2}\Blue{2}\Magenta{3}\Blue{2}\Blue{2}\Blue{2}\Magenta{3} 
& & & 419 & 2881 & 27328 & 89208 & 195921 & 84243 & 400000 \\\hline\hline
Total &{\!0\!}&{\!0\!} & 274505 & 98489 &{\!156429}&{\!203791}& 245374 & 621412 &{\!1600000} \\\hline
\end{tabular}

}
\end{table}
%%%%%%%%%%%%%%%%%%%%%%%%%%%%%%%%%%%%%%%%%%%%%%%%%%%%%%%%%%%%%%%%%%%%%%%%%%%%%%%%%%%%
This is the same problem studied in Table~\ref{T:W1^7X1^2=14} and the notation is the
same. 
These two tables are similar, except that the data in
Table~\ref{T2:W1^7X1^2=14} suggest a lower bound of four for the number of real
solutions.
This is however an illusion.
As with the Secant Conjecture, the Monotone Conjecture is a limiting case of the
Monotone-Secant Conjecture, and for any selection of osculating flags, there is a
sufficiently nearby selection of disjoint secant flags which occur in the same order.
Thus, from the computations in the last row of Table~\ref{T:W1^7X1^2=14}, we know there
exist disjoint secant flags with necklace 
$\Blue{2222}\Magenta{3}\Blue{222}\Magenta{3}$ having no real solutions, and disjoint
secant flags with two real solutions, even though these were not observed in the
experiment.

This idea shows that there should be \emph{fewer} restrictions on the
numbers of real solutions for secant flags than for osculating flags.
Typically, we observe that the tables for osculating and secant flags look basically the
same, with a few exceptions.
We do not understand why the tables are so similar, and why in some cases they differ
slightly.

%%%%%%%%%%%%%%%%%%%%%%%%%%%%%%%%%%%%%%%%%%%%%%%%%%%%%%%%%%%%%%%%%% 
\section{Lower bounds and gaps on the number of real
             solutions}\label{S:real_structure}   

By the Theorem of Mukhin, Tarasov, and Varchenko, any osculating instance of a 
Schubert problem in a Grassmannian with real osculation points has all solutions real.
The set of solutions forms a real variety,
but there are other ways for an osculating instance to define a real variety (e.g.\ some 
pairs of osculation points are complex conjugates).
Work of Eremenko and Gabrielov~\cite{EG02a} suggests that there may be lower bounds on the
numbers of real solutions to such real osculating 
instances of Schubert problems.
We explain the background, describe an experiment to study
this question of additional structure, and give some results that have been inspired by
this experimentation. 
This work formed part of the 2013 Ph.D.\ thesis of Nickolas Hein.

%%%%%%%%%%%%%%%%%%%%%%%%%%%%%%%%%%%%%%%%%%%%%%%%%%%%%%%%%%%%%%%%%% 
\subsection{Topological lower bounds}
Eremenko and Gabrielov~\cite{EG02a} 
considered the Wronski map~\eqref{Eq:Wronski_map} restricted to spaces of real
polynomials, 
 \begin{equation}\label{Eq:Real_Wr}
   \Wr_{\R}\ \colon\ \Gr(k,\R_{n-1}[t])\ \longrightarrow\ \P(\R_{k(n-k)}[t])\,.
 \end{equation}
These are real manifolds of dimension $k(n{-}k)$.
If they were oriented, the Wronski map would
have a well-defined topological degree which is computed on the fiber over any regular
value $\Psi\in\P(\R_{k(n-k)}[t])$,
\[
   \deg \Wr_\R\ =\ \sum_{H\in\Wr^{-1}(\Psi)} 
     \mbox{\rm sign}\, d\Wr_{\R}(H)\,,
\]
where $\mbox{\rm sign}\, d\Wr_{\R}(H)$ is $1$ if the orientations at $H$ and $\Psi$ agree
and $-1$ if they do not.
The point is that if the spaces in~\eqref{Eq:Real_Wr} were oriented so that $\deg\Wr_\R$
is defined, then $|\deg\Wr_\R|$ would be a lower bound on the number of real points in a
fiber above a regular point $\Psi$.

While the Grassmannian and projective space in~\eqref{Eq:Real_Wr} are often
not orientable,  they always have a double cover which is (the oriented Grassmannian
and the sphere), so this degree may be computed on the double cover and it gives a
lower bound as described.
Eremenko and Gabrielov more generally considered the Wronski map~\eqref{Eq:Real_Wr}
restricted to real Schubert varieties $X_\lambda\Fdot(\infty)$.
Soprunova and Sottile~\cite[Th.~6.4]{SoSo06} extended this to Richardson varieties 
$X_\lambda\Fdot(\infty)\cap X_\mu\Fdot(0)$.

Given a partition $\lambda$, let \defcolor{$\lambda^c$} be 
$ n{-}k{-}\lambda_k\geq\dotsb\geq n{-}k{-}\lambda_1$, the complement of its Young
diagram in the $k\times(n{-}k)$ rectangle.
For $\lambda=(3,1)$ and $k=3,n=8$, we have $\lambda^c=(5,4,2)$.
When $\mu\subset\lambda$, the \demph{skew diagram $\lambda/\mu$} is $\lambda$ with the
boxes of $\mu$ removed. 
For example, if
\[
   \lambda\ =\ \raisebox{-5pt}{\FiFT}
   \qquad\mbox{and}\qquad
   \mu\ =\ \raisebox{-2.5pt}{\TI}
   \qquad\mbox{then}\qquad
   \lambda/\mu\ =\ \raisebox{-5pt}{\includegraphics{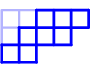}}\ .
\]
When $\mu=(0)$, we have $\lambda/\mu=\lambda$.

A \demph{Young tableau} of \demph{shape} $\lambda/\mu$ is a filling of the boxes of
$\lambda/\mu$ 
with the integers $1,2,\dotsc,|\lambda|{-}|\mu|$ that increases across each row and down
each column.
The \demph{standard filling} is when the numbers are in reading order.
For example, here are four (of the 324) Young tableaux of shape $(5,5,2)/(3)$ and the first is
the standard filling.
\[
  \Skew{1}{2}{3}{4}{5}{6}{7}{8}{9}
   \qquad
  \Skew{6}{8}{1}{3}{5}{7}{9}{2}{4}
   \qquad
  \Skew{2}{4}{1}{5}{6}{7}{9}{3}{8}
   \qquad
  \Skew{3}{6}{1}{2}{4}{5}{9}{7}{8}
\]
Let \defcolor{$\YT(\lambda/\mu)$} be the set of all Young tableaux of shape $\lambda/\mu$.

Given partitions $\lambda,\mu$, the Wronski map restricts to give a map
 \begin{equation}\label{Eq:skew-Wronskian}
   \defcolor{\Wr_{\lambda,\mu}}\ \colon\ 
    X_{\lambda}\Fdot(\infty)\cap X_{\mu}\Fdot(0)
    \ \longrightarrow\ 
    \P\bigl(t^{|\mu|}\C_{k(n{-}k)-|\lambda|-|\mu|}[t]\bigr)
 \end{equation}
whose degree is equal to the number  of Young tableaux of shape
$\lambda^c/\mu$. 
Thus when $k=3$, $n=8$, and $\lambda=\mu=\Th$, the degree of $\Wr_{\lambda,\mu}$ is 324.
%
%   for Gr(3,8), here are the most extreme sign imbalances
%                      sigma deg
% [1, 2, 6], [3, 5, 8],  3,  105
% [1, 2, 6], [3, 7, 8],  4,  324
% [1, 3, 6], [3, 5, 8],  3,  75
% [1, 3, 6], [4, 5, 8],  3,  105
% [1, 4, 5], [3, 6, 8],  3,  105
% [1, 4, 6], [3, 6, 8],  3,  75
% [1, 4, 6], [3, 7, 8],  3,  105
%
Restricting to the real points, the map~\eqref{Eq:skew-Wronskian} becomes a map between
the real Richardson variety and the real projective space.
Lifting to double covers as before, it has a topological degree 
(the singularities of the Richardson variety cause no harm, as they are in
codimension 2). 

Every tableau $\calT\in\YT(\lambda^c/\mu)$ has a sign, \defcolor{$\mbox{sgn}(\calT)$},
which is the sign of the permutation mapping the standard filling to $\calT$.
The \demph{sign-imbalance $\sigma(\lambda^c/\mu)$} of $\YT(\lambda^c/\mu)$ is 
defined to be
\[
   \sigma(\lambda^c/\mu)\ :=\  
    \biggl|\sum_{\calT\in\YT(\lambda^c/\mu)} \mbox{sgn}(\calT) \biggr|\,.
\]
For $\lambda=\mu=\Th$, we have 
$\sigma((552)/(3))=\sigma(\includegraphics{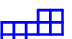})=4$.

%%%%%%%%%%%%%%%%%%%%%%%%%%%%%%%%%%%%%%%%%%%%%%%%%%%%%%%%%%%%%%%%%%%%%%%%%%%%%%%%%
\begin{theorem}[\cite{EG02a,SoSo06}]\label{Th:sign-imbalance}
 The restriction $\Wr_{\lambda,\mu}$~$\eqref{Eq:skew-Wronskian}$
 of the real Wronski map has topological degree equal to 
 the sign-imbalance $\sigma(\lambda^c/\mu)$.
\end{theorem}
%%%%%%%%%%%%%%%%%%%%%%%%%%%%%%%%%%%%%%%%%%%%%%%%%%%%%%%%%%%%%%%%%%%%%%%%%%%%%%%%%

This sign-imbalance is thus a \demph{topological lower bound} on the number of real points
in a fiber of the real Wronski map.
It is not known in general whether this lower bound is attained.
Eremenko and Gabrielov showed that when $\lambda=\mu=\emptyset$, the topological 
lower bound is positive when $n$ is odd and it is zero when $n$ is even~\cite{EG02a}.
Later, they showed that when both $k$ and $n$ are even, the map $\Wr_\R$ is not
surjective~\cite{EG02b}, so the topological lower bound of zero is attained when both $k$
and $n$ are even.
All other cases remained open.

%%%%%%%%%%%%%%%%%%%%%%%%%%%%%%%%%%%%%%%%%%%%%%%%%%%%%%%%%%%%%%%%%%%%%%%%%%%%%%%%%
\subsection{Real osculating instances of Schubert problems}
An osculating instance of a Schubert problem
 \begin{equation}\label{Eq:OISP}
   X\ =\    X_{\lambda^1}\Fdot(t_1) \,\cap\,
   X_{\lambda^2}\Fdot(t_2) \,\cap\; \dotsb  \;\cap\,
   X_{\lambda^r}\Fdot(t_r) 
 \end{equation}
is \demph{real} if $X$ equals its complex conjugate $\overline{X}$.
By Proposition~\ref{P:useful}, this means that for each $i=1,\dotsc,r$, either $t_i$ is
real or there is a unique $j\neq i$ with $\lambda^j=\lambda^i$ and $t_j=\overline{t_i}$.

The \demph{type} of a real osculating instance~\eqref{Eq:OISP} of a Schubert problem
$\lambda^1,\dotsc,\lambda^r$ is the list 
$(\rho_\lambda\mid \lambda\in\{\lambda^1,\dotsc,\lambda^r\})$ where $\rho_\lambda$ is the
number of indices $i$ with $t_i$ real and $\lambda=\lambda^i$.
For example, the Schubert problem
 \[
    X_{\Is}\Fdot(0)\,\cap\,X_{\Is}\Fdot(\infty)\,\cap\,
    X_{\Is}\Fdot(1)\,\cap\,X_{\TIs}\Fdot(\sqrt{-1})\,\cap\,
    X_{\TIs}\Fdot(-\sqrt{-1})
 \]
in $\Gr(3,6)$ has type $(\rho_{\,\Is},\rho_{\,\TIs})=(3,0)$.
The Theorem of Mukhin, Tara\-sov, and Varchenko involves real osculating instances of
maximal type, where $\rho_\lambda=\#\{i\mid \lambda^i=\lambda\}$.

The experimentation and results we describe shed light on structures in
the possible numbers of real solutions to real osculating instances of Schubert problems
which may depend on type.

%%%%%%%%%%%%%%%%%%%%%%%%%%%%%%%%%%%%%%%%%%%%%%%%%%%%%%%%%%%%%%%%%%%%%%%%%%%%%%%%%
\subsection{Experiment}
Hein, Hillar, and Sottile set up and ran a computational experiment using the framework 
of the projects~\cite{FRSC_secant,FRSC_monotone,Exp-FRSC} described in
Section~\ref{S:generalizations} to investigate structure in the
numbers of real solutions to real osculating instances of Schubert problems that depend
upon the osculation type of the instance.
This followed the broad outline of those projects, with some differences.
These differences included that it was mostly run on the brazos cluster at Texas A\&M
and did not use Maple, relying instead on Singular's~\cite{Singular} nrroots command
from the rootsur~\cite{rootsur_lib} library which 
computes the numbers of real roots of a real univariate polynomial.

Since Singular (and symbolic software in general) does not perform efficiently over
the field $\Q[\sqrt{-1}]$, this experiment used a slightly different formulation of
Schubert problems than indicated in Subsection~\ref{SS:SC}.

%%%%%%%%%%%%%%%%%%%%%%%%%%%%%%%%%%%%%%%%%%%%%%%%%%%%%%%%%%%%%%%%%%%%%%%%%%%%%%%%%
\begin{proposition}
 Suppose that $t\in\C$ is not real and $S$ is the collection of minors of
 matrices $\left[\begin{smallmatrix}I_k:X\\F_i(t)\end{smallmatrix}\right]$ that define
 $X_\lambda\Fdot(t)$ in the local coordinates $[I_k:X]$.
 The intersection $X_\lambda\Fdot(t)\cap X_\lambda\Fdot(\overline{t})$
 is defined in $[I_k:X]$ by the real polynomials
\[
   \{  \Re(f)+\Im(f) \mid f\in S\}\,,
\]
 the real and imaginary parts of polynomials in $S$.
\end{proposition}
%%%%%%%%%%%%%%%%%%%%%%%%%%%%%%%%%%%%%%%%%%%%%%%%%%%%%%%%%%%%%%%%%%%%%%%%%%%%%%%%%

The data from the experiment are available on line~\cite{Lower_Exp} and are presented as
before in frequency tables for each Schubert problem showing the observed numbers of real
solutions for osculating intersections of a given type.
For example, Table~\ref{Table:51e7=6} shows the frequency table for 
the problem $\Fi\cdot\I^7$ with six solutions.
%%%%%%%%%%%%%%%%%%%%%%%%%%%%%%%%%%%%%%%%%%%%%%%%%%%%%%%%%%%%%%%%%%%%%%%%%%%%%%%%%
\begin{table}[htb]
 \caption{Frequency table for $\Fi\cdot\I^7=6$ in $\Gr(2,8)$.}\vspace{-10pt}
 \label{Table:51e7=6}
 {\begin{tabular}{|c||r|r|r|r||r|}
   \multicolumn{6}{c}{Number of Real Solutions}\\\hline
   $\rho_{\,\Is}$   &{0}&{2}&{4}&{6}& Total\\\hline\hline
  {1}&8964&67581&22105&1350&{100000}\\\hline
  {3}&&47138&47044&5818&{100000}\\\hline
  {5}&&&77134&22866&{100000}\\\hline
  {7}&&&&100000&{100000}\\\hline\hline
  {Total}&$8964$&{114719}&{146283}&
   {130034}&{400000}\\\hline
 \end{tabular}}

\end{table}
%%%%%%%%%%%%%%%%%%%%%%%%%%%%%%%%%%%%%%%%%%%%%%%%%%%%%%%%%%%%%%%%%%%%%%%%%%%%%%%%%
Note that $\rho_{\,\Is}{-}1$ is the apparent lower bound for the minimal number of real
solutions as a function of the type $\rho_{\,\Is}$.
The observed lower bound of zero for this Schubert problem is the sign-imbalance of 
Theorem~\ref{Th:sign-imbalance}, as $\Fi^c$ has sign-imbalance
zero.

This experiment studied 756 Schubert problems, 273 of which had a topological lower bound
given by Theorem~\ref{Th:sign-imbalance}.
These included the Wronski maps for $\Gr(2,4)$, $\Gr(2,6)$, and $\Gr(2,8)$ for which
Eremenko and Gabrielov had shown the lower bound of zero was sharp.
For 264 of the remaining 270 cases, the sharpness of the topological lower bound was
verified.
There were however six Schubert problems for which the topological lower bounds were not
observed.
These were
\[
  \bigl(\,\raisebox{-2.5pt}{\ThTh},\raisebox{-5pt}{\III},\I^7\bigr)\,,\ 
  \Bigl(\,\raisebox{-5pt}{\TTT},\Th,\I^7\Bigr)\,,\ 
  \bigl(\,\raisebox{-2.5pt}{\ThTh},\raisebox{-2.5pt}{\TT},\I^6\bigr)\,,\ 
  \Bigl(\,\raisebox{-5pt}{\TTT},\raisebox{-2.5pt}{\TT},\I^6\Bigr)\,,\ 
  \Bigl(\,\raisebox{-5pt}{\ThThT},\I^8\Bigr)\,,
\]
all in $\Gr(4,8)$, and $\I^9$ in $\Gr(3,6)$.
These have observed lower bounds of $3,3,2,2,2,2$ and sign-imbalances of $1,1,0,0,0,0$,
respectively. 
There is not yet an explanation for the first four, but the last two are 
symmetric Schubert problems, which were observed to have a
congruence modulo four on their numbers of real solutions.
This congruence gives a lower bound of two for both problems $\ThThTs\cdot\I^8$ in
$\Gr(4,8)$ and $\I^9$ in $\Gr(3,6)$.
This will be discussed in Subsection~\ref{SS:SSP}.

There is another family of Schubert problems containing the problem
$\Fi\cdot\I^7$ in $\Gr(2,8)$ of Table~\ref{Table:51e7=6}.
This family has one problem in each Grassmannian $\Gr(k,n)$---it 
involves a large rectangular partition and the partition $\I$ repeated $n{-}1$
times, such as $\ThThs\cdot\I^6$ in $\Gr(3,7)$.
Each Schubert problem in this family has lower bounds which depend upon
$\rho_{\Is}$, as well as gaps in the possible numbers of real solutions.
This is discussed in Subsection~\ref{SS:lbg}.

In addition to these families of Schubert problems, this experiment~\cite{Lower_Exp}
found many Schubert problems with apparent additional structure to their numbers
of real solutions to real osculating instances.
However, the data did not suggest any other clear conjectures that would explain most of the
observed structure. 
Table~\ref{Table:33.2e5=10} shows another frequency table from this experiment.
%%%%%%%%%%%%%%%%%%%%%%%%%%%%%%%%%%%%%%%%%%%%%%%%%%%%%%%%%%%%%%%%%%%%%%%%%%%%%%%%%
\begin{table}[htb]
 \caption{Frequency table for $\ThThs\cdot\IIs^5=10$ in $\Gr(4,8)$.}\vspace{-10pt}
  \label{Table:33.2e5=10}
 {\begin{tabular}{|c||r|r|r|r|r|r||r|}
   \multicolumn{8}{c}{{Number of Real Solutions}}\\\hline
   $\rho_{\,\IIs}$&{0}&{2}&{4}&{6}&{8}&{10}&Total\\\hline\hline
  {1}&&138225&49674&2077&5404&4620&{200000}\\\hline
  {3}&&&163693&6458&8142&21707&{200000}\\\hline
  {5}&&&&&&200000&{200000}\\\hline\hline
  {Total}&&{138225}&{213367}&
   {8535}&{1356}&{226327}&{600000}\\\hline
 \end{tabular}}

\end{table}
%%%%%%%%%%%%%%%%%%%%%%%%%%%%%%%%%%%%%%%%%%%%%%%%%%%%%%%%%%%%%%%%%%%%%%%%%%%%%%%%%

%%%%%%%%%%%%%%%%%%%%%%%%%%%%%%%%%%%%%%%%%%%%%%%%%%%%%%%%%%%%%%%%%%%%%%%%%%%%%%%%%
\subsection{Symmetric Schubert problems}\label{SS:SSP}
Partitions for $\Gr(k,2k)$ are subsets of a $k\times k$ square.
A partition $\lambda$ is \demph{symmetric} if it equals its matrix transpose.
All except the last of the following partitions are symmetric,
\[
  \I\,,\ 
  \raisebox{-2.5pt}{\TI}\,,\ 
  \raisebox{-2.5pt}{\TT}\,,\ 
  \raisebox{-5pt}{\ThII}\,,\ 
  \raisebox{-5pt}{\ThTI}\,,\ 
  \raisebox{-5pt}{\ThThT}\,,\ 
  \raisebox{-7.5pt}{\FThThI}\,,\ 
  \raisebox{-5pt}{\FiFT}\ .
\]
For a symmetric partition $\lambda$, let $\ell(\lambda)$ be the number of boxes on its
main diagonal, which is the maximum number $i$ with $i\leq\lambda_i$.
Thus $\ell(\I)=\ell(\TIs)=1$ and $\ell(\TTs)=2$.
A Schubert problem $\blambda=(\lambda^1, \dotsc, \lambda^r)$ is \demph{symmetric} if each
partition $\lambda^i$ is symmetric.
The numbers of real solutions to real osculating instances of many symmetric Schubert
problems obey a congruence modulo four.

%%%%%%%%%%%%%%%%%%%%%%%%%%%%%%%%%%%%%%%%%%%%%%%%%%%%%%%%%%%%%%%%%%%%%%%%%%%%%%%%%
\begin{theorem}\label{Th:mod_four}
 Let $\blambda=(\lambda^1,\dotsc,\lambda^r)$ be a symmetric Schubert problem with 
 $\sum_i\ell(\lambda^i)>k{+}3$.
 Then the number of real solutions to any real osculating
 instance of $\blambda$ is congruent to $d(\blambda)$ modulo four.
\end{theorem}
%%%%%%%%%%%%%%%%%%%%%%%%%%%%%%%%%%%%%%%%%%%%%%%%%%%%%%%%%%%%%%%%%%%%%%%%%%%%%%%%%

A weak version of Theorem~\ref{Th:mod_four} was proven in~\cite{Modfour}, where the full
statement was conjectured, and finally proved
in~\cite{new_modfour}.
The fundamental idea is that the Grassmannian $\Gr(k,2k)$ has an
algebraic Lagrangian involution which restricts to an involution on the solutions to any
osculating instance of the symmetric Schubert problem $\blambda$.
This involution commutes with complex conjugation on the set of solutions to a real
osculating instance.
When a codimension condition on the fixed points is satisfied, the interaction of these
two involutions implies the congruence modulo four.
Before sketching the main ideas in the proof of Theorem~\ref{Th:mod_four}, we give an
interesting Corollary.

%%%%%%%%%%%%%%%%%%%%%%%%%%%%%%%%%%%%%%%%%%%%%%%%%%%%%%%%%%%%%%%%%%%%%%%%%%%%%%%%%
\begin{corollary}
 Real osculating instances of the Schubert problem $\I^9$ in $\Gr(3,6)$ with $42$
 solutions always have at least two real solutions, counted with multiplicity.
\end{corollary}
%%%%%%%%%%%%%%%%%%%%%%%%%%%%%%%%%%%%%%%%%%%%%%%%%%%%%%%%%%%%%%%%%%%%%%%%%%%%%%%%%

Thus, the topological lower bound of zero for this Schubert problem is not sharp.
A similar lack of sharpness holds for the symmetric Schubert problem
$(\ThThTs,\I^8)=90$ in $\Gr(4,8)$.

%%%%%%%%%%%%%%%%%%%%%%%%%%%%%%%%%%%%%%%%%%%%%%%%%%%%%%%%%%%%%%%%%%%%%%%%%%%%%%%%%
Let $\langle\cdot,\cdot\rangle$ be a symplectic (non-degenerate and skew-symmetric) form 
on $\C^{2k}$, which we may assume is
\[
  \langle {\bf e}_i,{\bf e}_{2k+1-j}\rangle\ =\ (-1)^i\delta_{i,j}\,,
\]
where ${\bf e}_1,\dotsc,{\bf e}_{2k}$ is the standard basis for $\C^{2k}$.
Any subspace $V\subset\C^{2k}$ has an annihilator $V^\perp$ under 
$\langle\cdot,\cdot\rangle$,
\[
   \defcolor{V^\perp}\ :=\ \{ w\in\C^{2k}\,\mid\, 
    \langle w,v\rangle=0\quad\forall v\in  V\}\,.
\]
As $\langle\cdot,\cdot\rangle$ is non-degenerate, we have $\dim V+\dim V^\perp=2k$ and 
$(V^\perp)^\perp=V$, so that $\perp$ is an involution on the set of linear subspaces of
$\C^{2k}$ which restricts to an involution on the Grassmannian $\Gr(k,2k)$.
We have $(X_\lambda\Fdot)^\perp=X_{\lambda^T}\Fdot^\perp$, where $\lambda^T$ is the matrix
transpose of $\lambda$.
The \demph{Lagrangian Grassmannian $\LG(k)$} is the set of points of $\Gr(k,2k)$ that are
fixed under this involution.

Given a flag $\Edot\colon E_{a_1}\subset\dotsb\subset E_{a_p}$ its annihilators
$E_{a_p}^\perp\subset\dotsb\subset E_{a_1}^\perp$ form a flag $\Edot^\perp$.
If $\defcolor{\Fl(2k)}:=\Fl(\{1,2,\dotsc,2k{-}1\};2k)$ is the manifold of complete flags,
then  $\Fdot\mapsto \Fdot^\perp$ is an involution on $\Fl(2k)$ whose fixed points are 
\demph{symplectic flags}, the flag manifold for the symplectic group which preserves the
form $\langle\cdot,\cdot\rangle$.

The Lagrangian Grassmannian has Schubert varieties $Y_\lambda\Fdot$ which are given by
a symmetric partition $\lambda$ and a symplectic flag.
The Schubert variety $Y_\lambda\Fdot$ is the set of fixed points of the Lagrangian involution
acting on $X_\lambda\Fdot$.  
The dimension of $\LG(k)$ is $\binom{k+1}{2}$ and the codimension of
$Y_\lambda\Fdot$ is $\defcolor{\|\lambda\|}:=\frac{1}{2}(\ell(\lambda)+|\lambda|)$.

If we take our rational normal curve to be 
 \begin{equation}\label{Eq:lagr_curve}
   \defcolor{\gamma(t)}\  :=\ (1\,,\, t\,,\, \tfrac{1}{2}t^2 \,,\, \tfrac{1}{6}t^3 
        \,,\, \dotsc \,,\, \tfrac{1}{(2k-1)!}t^{2k-1})\,,
 \end{equation}
then the flags $\Fdot(t)$ osculating $\gamma$ are symplectic.
(This comes from a regular nilpotent as in Subsection~\ref{S:other}.)

Theorem~\ref{Th:mod_four} is a consequence of a more general result.
An instance 
\[
   X_{\lambda^1}\Fdot^1\,\cap\,
   X_{\lambda^2}\Fdot^2\,\cap\;\dotsb\;\cap\,
   X_{\lambda^r}\Fdot^r
\]
of a Schubert problem is \demph{real} if for every $i=1,\dotsc,r$ there is a $j$ with 
$\lambda^i=\lambda^j$ and $\overline{\Fdot^i}=\Fdot^j$.
The following is proven in~\cite{new_modfour}.

%%%%%%%%%%%%%%%%%%%%%%%%%%%%%%%%%%%%%%%%%%%%%%%%%%%%%%%%%%%%%%%%%%%%%%%%%%%%%%%%%
\begin{theorem}\label{Th:new_modfour}
 Suppose that $\blambda$ is a symmetric Schubert problem with 
 $\sum_i \ell(\lambda^i)> k{+}3$.
 Then for any real instance of the Schubert problem $\blambda$
 \begin{equation}\label{Eq:symm_instance}
   X_{\lambda^1}\Fdot^1\,\cap\,
   X_{\lambda^2}\Fdot^2\,\cap\;\dotsb\;\cap\,
   X_{\lambda^r}\Fdot^r
 \end{equation}
 where the flags $\Fdot^i$ are symplectic, 
 the number of real solutions is congruent to $d(\blambda)$ modulo four.
\end{theorem}
%%%%%%%%%%%%%%%%%%%%%%%%%%%%%%%%%%%%%%%%%%%%%%%%%%%%%%%%%%%%%%%%%%%%%%%%%%%%%%%%%

Theorem~\ref{Th:mod_four} is the special case of Theorem~\ref{Th:new_modfour} when
the symplectic flags are osculating.
The proof rests on a simple lemma from~\cite{Modfour}.

Suppose that $f\colon Y\to Z$ is a proper dominant map between complex algebraic varieties
with $Z$ smooth, and that $Y$, $Z$, and $f$ are all defined over $\R$.
The degree $d$ of $f$ is the number of complex points of $Y$ above any regular value
$z\in Z$ of $f$.
If $z\in Z(\R)$ is a real regular value, then the number of real points in $f^{-1}(z)$ is
congruent to $d$ modulo two.
Suppose that the variety $Y$ has an involution $\iota\colon Y\to Y$ that preserves the
fibers of $f$, so that $f(y)=f(\iota(y))$.
We have the following.

%%%%%%%%%%%%%%%%%%%%%%%%%%%%%%%%%%%%%%%%%%%%%%%%%%%%%%%%%%%%%%%%%%%%%%%%%%%%%%%%%
\begin{lemma}\label{L:simple}
 If the image $f(Y^\iota)$ of the fixed points of\/ $Y$ under $\iota$ has codimension at
 least two in $Z$, then for any real regular values $z,z'$ of $f$ lying in the same
 connected component of $Z(\R)$, the number of real points in $f^{-1}(z)$ is congruent to
 the number of real points in $f^{-1}(z')$, modulo four.
\end{lemma}
%%%%%%%%%%%%%%%%%%%%%%%%%%%%%%%%%%%%%%%%%%%%%%%%%%%%%%%%%%%%%%%%%%%%%%%%%%%%%%%%%

Lemma~\ref{L:simple} is applied to a universal family $Y_{\blambda}\to Z_{\blambda}$ of
instances of a symmetric Schubert problems $\blambda$. 
The base $Z_{\blambda}$ consists of $r$-tuples of symplectic flags
$(\Fdot^1,\dotsc,\Fdot^r)$ where if $\lambda^i=\lambda^j$, then the order of $\Fdot^i$ and
$\Fdot^i$ does not matter.
(Technically, $Z_{\blambda}$ is the product of appropriate symmetric products of the
manifold of symplectic flags.)

The fiber of $Y_{\blambda}$ over a point $z=(\Fdot^1,\dotsc,\Fdot^r)$ of $Z_{\blambda}$ is
the intersection~\eqref{Eq:symm_instance}.
The main result of~\cite{General} states that when the point $z\in Z_{\blambda}$ is
general, the intersection~\eqref{Eq:symm_instance} is zero-dimensional, and thus
$Y_{\blambda}$ has the same dimension as $Z_{\blambda}$.
(This does not follow from Kleiman's Transversality Theorem~\cite{Kl74} as general
symplectic flags are not general flags.)

The points of $Y_{\blambda}$ fixed by the Lagrangian involution are
intersections of the corresponding Schubert varieties in the Lagrangian Grassmannian.
Kleiman's Theorem applies to those Lagrangian Schubert varieties which implies that
the fixed points $Y^\iota_{\blambda}$ have codimension 
\[
   \sum_{i=1}^r \|\lambda^i\|\ -\ \binom{k{+}1}{2}\ =\ 
    \frac{1}{2}\Bigl(\sum_{i=1}^r \ell(\lambda^i)\ -\ k\Bigr)\,.
\]
The condition $\sum_i\ell(\lambda^i)>k{+}3$ implies that this codimension is at least two,
so Lemma~\ref{L:simple} applies.
Lastly, the real points of $Z_{\blambda}$ are connected and the Theorem of Mukhin, Tarasov,
and Varchenko gives points of $Z_{\blambda}(\R)$ with all $d(\blambda)$ solutions real,
which implies Theorems~\ref{Th:new_modfour} and~\ref{Th:mod_four}.

%%%%%%%%%%%%%%%%%%%%%%%%%%%%%%%%%%%%%%%%%%%%%%%%%%%%%%%%%%%%%%%%%%%%%%%%%%%%%%%%%
\subsection{Lower bounds and gaps}\label{SS:lbg}

Table~\ref{Ta:333.1e7=20} shows the result of computing 800000 osculating instances of
the symmetric Schubert problem  $\ThThThs\cdot \I^7$ on  $\Gr(4,8)$ with 20 solutions, and 
tabulating the number of observed real solutions for osculating instances of a given
type. 
%
%
%%%%%%%%%%%%%%%%%%%%%%%%%%%%%%%%%%%%%%%%%%%%%%%%%%%%%%%%%%%%%%%%%%%%%%%%%%%%%%%%%
\begin{table}[htb]
 \caption{Gaps and lower bounds for $\ThThThs\cdot \I^7=20$ in $\Gr(4,8)$}\vspace{-10pt}
  \label{Ta:333.1e7=20}

{ \begin{tabular}{|c||r|c|r|c|r|c|r||r|}
   \multicolumn{9}{c}{{Number of Real Solutions}}\\\hline
   $\rho_{\Is}$&${0}$&${2}$&${4}$&${6}$
   &${8}$&$\dotsb$
   &${20}$&Total\\\hline\hline
   {1}&37074&&47271&&14517&$\dotsb$&1138&{100000}\\\hline
   {3}&&&66825&&30232&$\dotsb$&2943&{100000}\\\hline
   {5}&&&&&85080&$\dotsb$&14920&{100000}\\\hline
   {7}&&&&&&$\dotsb$&100000&{100000}\\\hline\hline
   {Total}&{37074}&&{114096}&&
       {129829}&$\dotsb$&{119001}&{400000}\\\hline
 \end{tabular}}
\end{table}
%%%%%%%%%%%%%%%%%%%%%%%%%%%%%%%%%%%%%%%%%%%%%%%%%%%%%%%%%%%%%%%%%%%%%%%%%%%%%%%%%
The ellipses $\dotsb$ mark columns (numbers of real solutions) that were not observed.
The hypotheses of Theorem~\ref{Th:mod_four} hold, so the number of real solutions is
congruent to 20 modulo 4.
The lack of instances with 12 and 16 real solutions, and the triangular shape of the
rest of the table (similar to the triangular shape of Table~\ref{Table:51e7=6}) are
additional structures which we explain.

These Schubert problems $\ThThThs\cdot \I^7$ and $\Fi\cdot\I^7$ are members of a family of
Schubert problems whose osculating instances we may solve completely and thereby determine
all possibilities for their numbers of real solutions.
Details are given in~\cite{Lower}.

For $k,n$, let \defcolor{$\Bx_{k,n}$} ($\Bx$ for short) be the partition consisting of
$k{-}1$ parts, each of size $n{-}k{-}1$.
For example, 
\[
   \Bx_{2,8}\ =\ \Fi\,,\quad 
   \Bx_{3,7}\ =\ \raisebox{-2.5pt}{\ThTh}\,,
   \quad\mbox{and}\quad
   \Bx_{4,8}\ =\ \raisebox{-5pt}{\ThThTh}\,.
\]
The osculating Schubert problems in this family have the form 
$\blambda=(\Bx,\I^{n-1})$ in $\Gr(k,n)$, and they all have topological lower bounds of
Theorem~\ref{Th:sign-imbalance}. 
The multinomial coefficient \defcolor{$\binom{n}{a,b}$} is zero unless $n=a{+}b$, and in
that case it equals $\frac{n!}{a!b!}$. 

%%%%%%%%%%%%%%%%%%%%%%%%%%%%%%%%%%%%%%%%%%%%%%%%%%%%%%%%%%%%%%%%%%%%%%%%%%%%%%%%%
\begin{lemma}\label{L:couting_factorizations}
 For the Schubert problem $\blambda = (\Bx,\I^{n-1})$ in $\Gr(k,n)$, we have 
 $d(\blambda)=\binom{n-2}{k-1}$ and 
 $\sigma(\Bx^c)=\binom{\lfloor\frac{n-2}{2}\rfloor}{\lfloor\frac{k-1}{2}\rfloor,
   \lfloor\frac{n{-}k{-}1}{2}\rfloor}$, 
 which is zero unless $n$ is even and $k$ is odd.
\end{lemma}
%%%%%%%%%%%%%%%%%%%%%%%%%%%%%%%%%%%%%%%%%%%%%%%%%%%%%%%%%%%%%%%%%%%%%%%%%%%%%%%%%

We show that Schubert problems reduce to finding factorizations $f=gh$ of univariate
polynomials.
For this, we will regard another factorization $f=g_1h_1$ where $g_1$ is a scalar multiple
of $g$ (and the same for $h_1$ and $h$), to be equivalent to $f=gh$.

%%%%%%%%%%%%%%%%%%%%%%%%%%%%%%%%%%%%%%%%%%%%%%%%%%%%%%%%%%%%%%%%%%%%%%%%%%%%%%%%%
\begin{theorem}\label{Th:factorization}
 For any $k,n$, the solutions to the osculating instance of the Schubert problem
 $(\Bx,\I^{n-1})$ in $\Gr(k,n)$
 \begin{equation}\label{Eq:osc_inst}
   X_{\Is}(t_1)\,\cap\,
   X_{\Is}(t_2)\,\cap\;\dotsb\;\cap\,
   X_{\Is}(t_{n-1})\,\cap\,
   X_{\Bxs}(\infty)
 \end{equation}
 may be identified with all ways of factoring $f'(t)=g(t)h(t)$ where
 \begin{equation}\label{Eq:eff}
   f(t)\ =\ \prod_{i=1}^{n-1} (t-t_i)
 \end{equation}
 with $\deg g=n{-}k{-}1$ and $\deg h = k{-}1$.
\end{theorem}
%%%%%%%%%%%%%%%%%%%%%%%%%%%%%%%%%%%%%%%%%%%%%%%%%%%%%%%%%%%%%%%%%%%%%%%%%%%%%%%%%

Thus the number of real solutions to a real osculating instance of the
Schubert problem $(\Bx,\I^{n-1})$ with osculation type $\rho_{\Is}$ is the number of
real factorizations $f'(t)=g(t)h(t)$ where $f(t)$ has exactly $\rho_{\Is}$ real roots,
$\deg g=n{-}k{-}1$, and $\deg h=k{-}1$.
This counting problem was studied in~\cite[Sect.~7]{SoSo06}, which we recall.
Let $\rho$ be the number of real roots of $f'(t)$.
By Rolle's Theorem, $\rho_{\Is}{-}1\leq \rho\leq n{-}2$.
Then the number \defcolor{$\nu(k,n,\rho)$} of such factorizations is the coefficient of
$x^{n-k-1}y^{k-1}$ in $(x+y)^\rho(x^2+y^2)^c$, where $c=\frac{n-2-\rho}{2}$, the number of
irreducible quadratic factors of $f'(t)$.

%%%%%%%%%%%%%%%%%%%%%%%%%%%%%%%%%%%%%%%%%%%%%%%%%%%%%%%%%%%%%%%%%%%%%%%%%%%%%%%%%
\begin{corollary}\label{C:factoring}
 The number of real solutions to a real osculating instance of the Schubert problem
 $(\Bx,\I^{n-1})$~\eqref{Eq:osc_inst} with osculation type $\rho_{\Is}$ is $\nu(k,n,\rho)$,
 where $r$ is the number of real roots of $f'(t)$, where $f$ is the
 polynomial~\eqref{Eq:eff}. 
\end{corollary}
%%%%%%%%%%%%%%%%%%%%%%%%%%%%%%%%%%%%%%%%%%%%%%%%%%%%%%%%%%%%%%%%%%%%%%%%%%%%%%%%%

%%%%%%%%%%%%%%%%%%%%%%%%%%%%%%%%%%%%%%%%%%%%%%%%%%%%%%%%%%%%%%%%%%%%%%%%%%%%%%%%
\begin{remark}
 When $\rho<n{-}4$, we have that $\nu(k,n,\rho)\leq\nu(k,n,\rho{+}2)$, so 
 $\nu(k,n,\rho_{\Is}{-}1)$ is the lower bound for the number of real solutions to a real
 osculating instance of $(\Bx,\I^{n-1})$ of osculation type $\rho_{\Is}$.
 Since at most $\lfloor\frac{n}{2}\rfloor$ different values of $r$ may occur for the
 numbers of real roots of $f'(t)$, but the number $\nu(k,n,\rho)$ satisfies
\[
  \binom{\lfloor\frac{n-2}{2}\rfloor}%
     {\lfloor\frac{k-1}{2}\rfloor,\lfloor\frac{n-k-1}{2}\rfloor}
    \ \leq\ \nu(k,n,\rho)\ \leq\ 
  \binom{n{-}2}{k{-}1}\ ,
\]
 There will in general be gaps in the possible numbers of real solutions, as we saw in
 Table~\ref{Ta:333.1e7=20}. 
 For example, the possible values of $\nu(5,13,\rho)$  are
\[
   10\,,\ 18\,,\ 38\,,\ 78\,,\ 162\,,\ \mbox{and}\ 330\,.
  \eqno{\qed}
\]
\end{remark}
%%%%%%%%%%%%%%%%%%%%%%%%%%%%%%%%%%%%%%%%%%%%%%%%%%%%%%%%%%%%%%%%%%%%%%%%%%%%%%%%%

%%%%%%%%%%%%%%%%%%%%%%%%%%%%%%%%%%%%%%%%%%%%%%%%%%%%%%%%%%%%%%%%%%%%%%%%%%%%%%%%%
\begin{proof}[Proof of Theorem~$\ref{Th:factorization}$]
 The Schubert variety $X_{\Bxs}(\infty)$ consists of the $k$-planes $H$ with
\[
   \dim H \cap F_{i+1}(\infty)\ \geq\ i\qquad\mbox{for } i=1,\dotsc,k{-}1\,.
\]
 By Proposition~\ref{P:useful}, the solutions to~\eqref{Eq:osc_inst} will be points in 
 $X_{\Bxs}(\infty)$ that do not lie in any other smaller Schubert variety
 $X_\lambda(\infty)$.
 This is the Schubert cell of $X_{\Bxs}(\infty)$~\cite{Fu97}, and it consists of the
 $k$-planes $H$ which are row spaces of matrices of the form
\[
  \left(\begin{array}{ccccccccc}
    1&x_1&\dotsb&x_{n-k-1}&x_{n-k}&0&\dotsb&0\\
    0&0  &\dotsb&   0   &  1   &x_{n-k+1}&\dotsb&0\\
   \vdots&\vdots&&\vdots&&\ddots &\ddots&\vdots\\
    0&0  &\dotsb&   0   &  \dotsb&0 & 1& x_{n-1}
    \end{array}\right)\ ,
\]
 where $x_1,\dotsc,x_{n-1}$ are indeterminates.
 If $x_{n-k}=0$, then $H\in X_{\Is}(0)$, but if one of $x_{n-k+1},\dotsc,x_{n-1}$
 vanishes, then $H\in X_{\IIs}(0)$, which cannot occur for a solution
 to~\eqref{Eq:osc_inst}, again by Proposition~\ref{P:useful}.
 Thus we may assume that $x_{n-k-1},\dotsc,x_{n-1}$ are non-zero.

 We use a scaled version of these coordinates.
 Consider variables Let $(f,g,h)$ be the $n{-}1$ variables
 $(f_0,\,g_0,\dotsc,g_{n-k-2},\,h_0,\dotsc,h_{k-2})$ with $h_0,\dotsc,h_{k-2}$ all
 non-zero, which we write as \defcolor{$(f,g,h)$}.
 Define constants  $\defcolor{g_{n-k-1}}:=1=:\defcolor{h_{k-1}}$  and
 $\defcolor{c_i}:=(-1)^{n-k-i+1}(n{-}k{-}i)!$.
 If \defcolor{$H(f,g,h)$} is the row space of the matrix (also written $H(f,g,h)$),
 \[
   \left(\begin{array}{cccccccccc}
     c_1g_{n-k-1} &  \dotsb & c_{n-k}g_0 &
         \frac{f_0}{h_0} & 0  &\dotsb&0&0\\
     0 &  \dotsb & 0     &-1 &\frac{h_0}{h_1}&\dotsb&0&0\\
     0 &  \dotsb & 0     & 0 &-2 &\ddots&\vdots&\vdots\\
    \vdots&&\vdots&\vdots&&\ddots&&0\\
     0 &   \dotsb & 0     & 0 &\dotsb&-(k{-}2)&\frac{h_{k-3}}{h_{k-2}} &0\\
     0 &   \dotsb & 0     & 0 &\dotsb&0&{\!-(k{-}1)\!}&\frac{h_{k-2}}{h_{k-1}}\\
   \end{array}\right)\ ,
 \]
then $H(f,g,h)$ parameterizes the Schubert cell of $X_{\Bxs}(\infty)$.
The following calculation is done in~\cite{Lower}.

%%%%%%%%%%%%%%%%%%%%%%%%%%%%%%%%%%%%%%%%%%%%%%%%%%%%%%%%%%%%%%%%%%%%%%%%%%%%%%%%%
\begin{lemma}\label{L:factorization}
 The condition for $H=H(f,g,h)\in X_{\Is}(t)$  is
 \begin{equation}\label{Eq:big_det}
  \det \left(\!\!\begin{array}{c}H\\F_{n-k}(t)\end{array}\!\!\right)
      \ =\ 
    (-1)^{k(n-k)}\Bigl(\, \sum_{i=0}^{n-k-1} \sum_{j=0}^{k-1} 
     \frac{t^{i+j+1}}{i{+}j{+}1} \, g_i\, h_j
      \ +\; f_0 \,\Bigr)\,.
 \end{equation}
\end{lemma}
%%%%%%%%%%%%%%%%%%%%%%%%%%%%%%%%%%%%%%%%%%%%%%%%%%%%%%%%%%%%%%%%%%%%%%%%%%%%%%%%%

Call this polynomial $f(t)$.
If $H$ lies in the intersection~\eqref{Eq:osc_inst}, then $f$ is the
polynomial~\eqref{Eq:eff}. 
If we set
 \begin{eqnarray*}
   g(t)&:=& g_0+tg_1+\dotsb+t^{n-k-1}g_{n-k-1} \quad\mbox{and}\\
   h(t)&:=&h_0+th_1+\dotsb+t^{k-1}h_{k-1}\,,
 \end{eqnarray*}
then $f(0)=f_0$ and $f'(t)=g(t)h(t)$.
Theorem~\ref{Th:factorization} is immediate.
\end{proof}
%%%%%%%%%%%%%%%%%%%%%%%%%%%%%%%%%%%%%%%%%%%%%%%%%%%%%%%%%%%%%%%%%%%%%%%%%%%%%%%%%

%%%%%%%%%%%%%%%%%%%%%%%%%%%%%%%%%%%%%%%%%%%%%%%%%%%%%%%%%%%%%%%%%%%%%%%%%%%%%%%%%
\section{Galois groups of Schubert problems}\label{S:Galois}
%%%%%%%%%%%%%%%%%%%%%%%%%%%%%%%%%%%%%%%%%%%%%%%%%%%%%%%%%%%%%%%%%%%%%%%%%%%%%%%%%

Not only do field extensions have Galois groups, but so do problems in enumerative
geometry, as Jordan explained in 1870~\cite{J1870}.
These algebraic Galois groups are identified with geometric monodromy groups.
While the earliest reference we know is Hermite in 1851~\cite{Hermite}, this point was
eloquently expressed by Harris in 1979~\cite{Ha79}. 
Jordan's treatise included examples of geometric problems, such as the 27 lines on a
cubic surface, whose (known) intrinsic structure prevents their Galois groups from being
the full symmetric group on their set of solutions.
In contrast, Harris's geometric methods enabled him to show that several classical
enumerative problems had the full symmetric group as their Galois group, and therefore had
no intrinsic structure.
Despite this, Galois groups are known for very few enumerative problems.

The first non-trivial computation of a Galois group in the Schubert calculus is due to
Byrnes and Stevens~\cite{BS_homotopy}. 
Interest in determining Galois groups of Schubert problems was piqued when
Derksen (see~\cite{Va06b}) discovered that the Schubert problem
$\TTs^4=6$ in $\Gr(4,8)$ has Galois group isomorphic to $S_4$ and is not the full
symmetric group $S_6$.
Ruffo et al.~\cite{RSSS} exhibited a Schubert problem in the flag manifold $\Fl(2,4;6)$
with six solutions whose Galois group was $S_3$ and not the full symmetric group $S_6$.
In both of these problems the intrinsic structure implies restrictions on the numbers of
real solutions. 
This is similar to the 27 lines on a real cubic surface, 
which may have either 3, 7, 15, or 27 real lines.

Vakil used the principle of specialization in enumerative geometry and group
theory to give a combinatorial method to obtain
information about Galois groups~\cite{Va06b}. 
Together with his geometric Littlewood-Richardson rule~\cite{Va06a} this gives a
recursive procedure that can show the Galois group of a Schubert problem contains the 
alternating group on its set of solutions.
This inspired Leykin and Sottile to show how numerical algebraic
geometry can be used to compute Galois groups~\cite{LS09}.
A third method based on elimination theory was proposed by Billey and Vakil~\cite{BV}. 
We discuss these three methods, including preliminary results and potential experimentation, 
explain how they were used to nearly determine the Galois groups of all Schubert 
problems in $\Gr(4,8)$ and $\Gr(4,9)$, and close with a description of two 
Schubert problems in $\Gr(4,8)$ whose Galois groups are not the full symmetric group.

%%%%%%%%%%%%%%%%%%%%%%%%%%%%%%%%%%%%%%%%%%%%%%%%%%%%%%%%%%%%%%%%%%%%%%%%%%%%%%%%%
\subsection{Galois groups}\label{S:Galois_groups}

Let $f\colon Y\to Z$ be a proper, generically separable and finite morphism of degree $d$,
where $Z$ and $Y$ are schemes of the same dimension with $Z$ smooth and $Y$ irreducible.
A point $z\in Z$ is a \demph{regular value} of $f$ if the fiber over $z$ consists of $d$ 
distinct points $\{y_1,\dotsc,y_d\}$. 
Write \defcolor{$S_d$} for the symmetric group on $d$ letters. 
Let \defcolor{$Y^{(d)}$} be the subscheme
\[
   \overline
   {(\underbrace{Y \times_Z \cdots \times_Z Y}_{d}) \setminus \Delta}\,,
\]
of the fiber product, where $\Delta$ is the big diagonal. 
Fixing a regular value $z\in Z$ of $f$ with $f^{-1}(z)=\{y_1,\dotsc,y_d\}$, 
the \demph{Galois/monodromy group  $\Gal{Y\to Z}$} is the group of permutations 
$\sigma\in S_d$ for which $(y_1,\dotsc,y_d)$ and $(y_{\sigma(1)},\dotsc,y_{\sigma(d)})$
lie on the same component  of $Y^{(d)}$. 
The  Galois group is well-defined up to conjugation in $S_d$.
As $Y$ is irreducible, $\Gal{Y\to Z}$ is transitive, and if 
$(Y\times_ZY)\smallsetminus\Delta$ is irreducible, then it 
is doubly transitive.

Fix a Schubert problem $\blambda=(\lambda^1,\dotsc,\lambda^r)$ in $\Gr(k,n)$.
Let $\Fl(n)$ be the manifold of complete flags in $\C^n$ and set 
$\defcolor{Z_{\blambda}}:=\prod_{i=1}^r \Fl(n)$, the $r$-fold product of flag manifolds,
which is smooth.
Set 
\[
  \defcolor{Y_{\blambda}}\ :=\  \{ (H; \Fdot^1,\dotsc,\Fdot^r)\mid 
   H\in X_{\lambda^i}\Fdot^i\,,\quad\mbox{for }i=1,\dotsc,r\}\,,
\]
the total space of the Schubert problem $\blambda$.
The projection $Y_{{\blambda}}\to\Gr(k,n)$ exhibits it as a fiber bundle with
fibers the product of Schubert varieties of $\Fl(n)$ of codimensions
$|\lambda^1|,\dotsc,|\lambda^r|$. 
Thus $Y_{\blambda}$ is irreducible, and, as $\blambda$ is a Schubert problem, 
$\dim Y_{\blambda}=\dim Z_{\blambda}$.

The fiber of $Y_{\blambda}$ over a point $(\Fdot^1,\dotsc,\Fdot^r)$ of $Z_{\blambda}$ is
the instance of the Schubert problem $\blambda$,
 \begin{equation}\label{Eq:fiber}
   X_{\lambda^1}\Fdot^1 \, \cap\,
   X_{\lambda^2}\Fdot^2 \, \cap\; \dotsb\; \cap\,
   X_{\lambda^r}\Fdot^r \,.
 \end{equation}
When $(\Fdot^1,\dotsc,\Fdot^r)$ is general, this is either empty or it consists of
finitely many points, by Kleiman's Theorem~\cite{Kl74}.
Thus $Y_{\blambda}\to Z_{\blambda}$ has a
Galois/monodromy group.
We call the Galois group $\calG_{Y_{\blambda}\to Z_{\blambda}}$ the
\demph{Galois group of the Schubert problem $\blambda$} and write 
\defcolor{$\calG_{\blambda}$} for it.

%%%%%%%%%%%%%%%%%%%%%%%%%%%%%%%%%%%%%%%%%%%%%%%%%%%%%%%%%%%%%%%%%%%%%%%%%%%%%%%%%
\subsection{Vakil's combinatorial criterion}\label{S:Vakils_criterion}

Vakil~\cite{Va06b} described how 
the monodromy group of the restriction to a subscheme $U\subset Z$ affects 
the Galois group $\Gal{Y\to Z}$.
Let $U\hookrightarrow Z$ be a closed embedding of a Cartier divisor, with $Z$ smooth
in codimension one along $U$.  
Consider the fiber diagram
 \begin{equation}\label{Eq:fiber_diagram}
  \raisebox{-20pt}{
  \begin{picture}(60,45)
   \put(5,35){$W$} \put(18,35){$\lhra$} \put(45,35){$Y$}
   \put(-3,19){$f$}\put(9,32){\vector(0,-1){20}}
      \put(48,32){\vector(0,-1){20}}\put(52,19){$f$}
   \put(5, 0){$U$} \put(18, 0){$\lhra$} \put(44, 0){$Z$}
  \end{picture}
  }
 \end{equation}
where $f\colon W\to U$ is generically finite and separable of degree $d$.
When $W$ is irreducible or has two components the following holds.
\begin{enumerate}
 \item[(a)] If $W$ is irreducible, then $\Gal{W\to U}$ includes into $\Gal{Y\to Z}$.
 \item[(b)] If $W$ has two components, $W_1$ and $W_2$, each of which maps dominantly to
   $U$ of respective degrees $d_1$ and $d_2$, then there is a subgroup $H$ of 
   $\Gal{W_1\to U}\times \Gal{W_2\to U}$ which maps surjectively onto each factor
   $\Gal{W_i\to U}$ and which includes into $\Gal{Y\to Z}$.
\end{enumerate}
A Galois group $\Gal{Y\to Z}$ is \demph{at least alternating} if it is either $S_d$ or
its alternating subgroup.  
In the above situation, Vakil gave criteria for deducing that $\Gal{Y\to Z}$ is at least
alternating, based on purely group-theoretic arguments including Goursat's Lemma.\medskip

%%%%%%%%%%%%%%%%%%%%%%%%%%%
\noindent{\bf Vakil's Criteria.}
{\it
  Suppose we have a fiber diagram as in~\eqref{Eq:fiber_diagram}. The Galois group 
  $\Gal{Y\to Z}$ is at least alternating if one of the following holds.
\begin{enumerate}
\item[(i)] In Case (a), if $\Gal{W\to U}$ is at least alternating.
\item[(ii)] In Case (b), if $\Gal{W_1\to U}$ and $\Gal{W_2\to U}$ are at least alternating
         and either $d_1\neq d_2$ or $d_1 = d_2 = 1$.
\item[(iii)] In Case (b), if $\Gal{W_1\to U}$ and $\Gal{W_2\to U}$ are at least
  alternating, one of $d_1$ or $d_2$ is not $6$, and $\Gal{Y\to Z}$ is  doubly transitive.
\end{enumerate}

 }\medskip
%%%%%%%%%%%%%%%%%%%%%%%%%%%%%%%%%%%%%%%%%%%%%%%%%%%%%%
Let $\Fdot, \Edot$ be flags in general position and $\lambda, \mu$ be partitions. 
Vakil's geometric Littlewood-Richardson rule is a sequence of degenerations that
convert the intersection  $X_\lambda \Fdot \cap X_\mu \Edot$ into a union of Schubert
varieties $X_{\nu}\Fdot$ with $|\nu| = |\lambda|+|\mu|$.
We write this as a formal sum
\begin{equation}\label{eq:LRR}
  X_{\lambda}\Fdot\cap X_{\mu}\Edot\ \sim\ \sum_{\nu} c^\nu_{\lambda,\mu}\, 
  X_{\nu}\Fdot \,,
\end{equation}
where $c_{\lambda,\mu}^\nu$ is the Littlewood-Richardson number. 
Each step from one degeneration to another is the specialization to a Cartier divisor $U$
in a family $Y\to Z$ representing the total space of the current degeneration as
in~\eqref{Eq:fiber_diagram}. 
These geometric degenerations and Vakil's criteria lead to a recursive algorithm to 
show that the Galois group of a Schubert problem is at least alternating, but which is not
a decision procedure---when the criteria fails, the Galois group may still be at least
alternating. 
Vakil wrote a maple script to apply this procedure (with criteria (i) and (ii)) to all
Schubert problems on a given Grassmannian. 

Vakil's method has been used to show that the Galois group of any
Schubert problem in $\Gr(2,n)$ is at least alternating.
We begin with some general definitions.
A \demph{special Schubert condition} is a partition $\lambda$ with only one non-zero part.
Write \defcolor{$a$} for the special Schubert condition $(a,0,\dotsc, 0)$.
A \demph{special Schubert problem} in $\Gr(k,n)$ is a list 
$\defcolor{\adot}:= (a_1,\dotsc, a_r)$
where $a_i>0$ and $|\adot|:=a_1+\dotsb+a_r=k(n{-}k)$.
Its number \defcolor{$d(\adot)$} of solutions is a Kostka number,
which counts the number of Young tableaux of shape $(n{-}k,\dotsc,n{-}k)=(n{-}k)^k$ of
content $\adot$. 

A special Schubert problem $\adot$ in $\Gr(2,n)$ has $a_i< n{-}2$ for all $i$.
It is \demph{reduced} if for all $i<j$ we have $a_i+a_j\leq n{-}2$, 
which implies that $r\geq 4$.
Any Schubert problem in $\Gr(2,n)$ is equivalent to a reduced special Schubert problem,
possibly on a smaller Grassmannian.

In 1884 Schubert~\cite{Sch1886a} gave a degeneration for special Schubert varieties in
$\Gr(2,n)$ that is a particular case of the geometric Littlewood-Richardson rule and
which may be used to decompose intersections in the same way as~\eqref{eq:LRR}.
Schubert's degeneration yields the recursion,
 \begin{equation}\label{eq:Schubert_recursion}
  \begin{split}
   d(a_1,\dotsc, a_r)\ &=\ d(a_1,\dotsc,a_{r-2}, a_{r-1}+a_{r}) \\ 
    &\quad \ +\  d(a_1,\dotsc,a_{r-2}, a_{r-1}{-}1, a_{r}{-}1)\,.
  \end{split}
 \end{equation}
Using this recursion,  Vakil's criterion (ii) implies that if each of the Schubert
problems 
$(a_1,\dotsc,a_{r-2}, a_{r-1}+a_{r})$ and $(a_1,\dotsc,a_{r-2}, a_{r-1}{-}1, a_{r}{-}1)$
on the right hand side of~\eqref{eq:Schubert_recursion} are at least alternating and if
the corresponding Kostka numbers are either distinct or are both equal to one, then 
the Galois group of the Schubert problem $\adot$ is at least alternating. 
Vakil used his maple script to check that all Schubert problems on 
$\Gr(2,n)$ for $n\leq 16$ were at least alternating, and 
Brooks, et al.~wrote their own script and extended Vakil's verification to $n\leq 40$.
Buoyed by these observations, Brooks, et al.~\cite{BMS12} proved the following theorem.

%%%%%%%%%%%%%%%%%%%%%%%%%%%%%%%%%%%%%%%%%%%%%%%%%%%%%%%%%%%%%%%%%%%%%%%%%%%%%%%%%
 \begin{theorem}\label{Th:BMS}
  Every Schubert problem in $\Gr(2,n)$ has Galois group that is at least alternating.
 \end{theorem}
%%%%%%%%%%%%%%%%%%%%%%%%%%%%%%%%%%%%%%%%%%%%%%%%%%%%%%%%%%%%%%%%%%%%%%%%%%%%%%%%%

The proof of Theorem~\ref{Th:BMS} is based on the following lemma.

%%%%%%%%%%%%%%%%%%%%%%%%%%%%%%%%%%%%%%%%%%%%%%%%%%%%%%%%%%%%%%%%%%%%%%%%%%%%%%%%%
\begin{lemma}
  Let $\adot$ be a reduced Schubert problem in $\Gr(2,n)$. 
  When $\adot \neq (1,1,1,1)$ there is a rearrangement $\adot = (a_1,\dotsc, a_r)$ with
 \begin{equation}\label{eq:inequality}
   d(a_1,\dotsc,a_{r-2}, a_{r-1}{+}a_{r})\ \neq\
   d(a_1,\dotsc,a_{r-2}, a_{r-1}{-}1, a_{r}{-}1)\,.
 \end{equation}
 When $\adot = (1,1,1,1)$, both terms of~\eqref{eq:inequality} are equal to $1$.
\end{lemma}
%%%%%%%%%%%%%%%%%%%%%%%%%%%%%%%%%%%%%%%%%%%%%%%%%%%%%%%%%%%%%%%%%%%%%%%%%%%%%%%%%

When some pair $a_i,a_j$ are unequal, the inequality~\eqref{eq:inequality} follows from a
combinatorial injection of Young tableaux. 
The remaining cases use that the Kostka numbers $d(\adot)$ are coefficients 
in the decomposition of tensor products of irreducible representations of $SU(2)$. 
Then the Weyl integral formula gives
 \begin{equation}
  d(a_1,\dotsc,a_r)\ =\  \frac{2}{\pi} \int_{0}^{\pi} 
       \Bigl(\prod_{i=1}^r \frac{\sin{(a_i{+}1)\theta}}{\sin{\theta}}\Bigr)\,
        \sin^2\theta\, d\theta\;.
 \end{equation}
Thus the inequality~\eqref{eq:inequality} is equivalent to showing that an integral is
non-zero, which is done by estimation.

Sottile and White~\cite{SW13} studied transitivity of Galois groups of Schubert 
problems, with an eye towards Vakil's Criterion (iii).
They showed that many Schubert problems has doubly transitive Galois groups.

%%%%%%%%%%%%%%%%%%%%%%%%%%%%%%%%%%%%%%%%%%%%%%%%%%%%%%%%%%%%%%%%%%%%%%%%%%%%%%%%%
 \begin{theorem}\label{Th:SW_double}
  Every Schubert problem in $\Gr(3,n)$, every special Schubert problem in $\Gr(k,n)$, and every
  simple Schubert problem of the form $(\lambda,\I^{k(n-k)-|\lambda|})$ in $\Gr(k,n)$
  has doubly transitive Galois group.
 \end{theorem}
%%%%%%%%%%%%%%%%%%%%%%%%%%%%%%%%%%%%%%%%%%%%%%%%%%%%%%%%%%%%%%%%%%%%%%%%%%%%%%%%%

They use the result for special Schubert problems and Vakil's criterion (iii) to give
another proof of Theorem~\ref{Th:BMS}. 

Our group plans to use Vakil's Criterion (ii) (and (iii) when double transitivity is known)
to study all Schubert problems on all small ($k(n{-}k)\lesssim 30$) Grassmannians.
The goal is to find Schubert problems whose Galois groups might not contain the alternating
group, and then use other methods to determine the Galois groups.
This approach has already been used for almost all Schubert problems in $\Gr(4,8)$ and
$\Gr(4,9)$, as we explain in Subsection~\ref{S:search}.
This experiment may involve several billion Schubert problems,
posing serious computer-science issues (such as storing the data or memory usage)
which must be resolved before it may begin.

%%%%%%%%%%%%%%%%%%%%%%%%%%%%%%%%%%%%%%%%%%%%%%%%%%%%%%%%%%%%%%%%%%%%%%%%%%%%%%%%%
\subsection{Homotopy continuation}\label{S:Homotopy}

By definition, the Galois group $\calG_{\blambda}$ of a Schubert problem $\blambda$ is the
monodromy group of the family $f\colon Y_{\blambda}\to Z_{\blambda}$, which may be understood
concretely as follows.
Regular values of the map $f$ are $r$-tuples of flags  $(\Fdot^1,\dotsc,\Fdot^r)$ for which 
the intersection~\eqref{Eq:fiber} is $d(\blambda)$ points.
Given a path $\gamma\colon[0,1]\to Z_{\blambda}$ consisting of regular values
of $f$, we may lift $\gamma$ to $Y_{\blambda}$ to obtain
$d(\blambda)$ paths connecting points in the fiber $f^{-1}(\gamma(0))$ to those in 
$f^{-1}(\gamma(1))$, inducing a bijection between these fibers.
When $\gamma$ is a loop based at a regular value $z$, we obtain a 
\demph{monodromy permutation} of the fiber $f^{-1}(z)$, and the set of all such 
monodromy permutations is the Galois group $\calG_{\blambda}$.
The computation of monodromy is feasible and is
an elementary operation in the field of numerical algebraic geometry.

Numerical algebraic geometry~\cite{SW05} uses numerical analysis to study algebraic
varieties on a computer.
It is based on Newton's method for refining approximate solutions to a system of equations
and its fundamental algorithm is path-continuation to follow solutions which depend
upon a real parameter $t\in[0,1]$.
Systems of equations are solved using \demph{homotopy methods}, which start with known
solutions to a system of equations and follow them along paths to obtain solutions to  the
desired system.
\demph{Parameter homotopy} is the most elementary; 
both the start and end systems have the same structure.
An example is the fibers of the map
$Y_{\blambda}\to Z_{\blambda}$ which are modeled by the determinantal equations of
Subsection~\ref{SS:SC}. 
There are more subtle and sophisticated homotopy methods that begin with solutions to 
simple systems of equations and bootstrap them to find all solutions to the desired
equations. 

This yields the following two-step procedure to compute monodromy permutations for a given
Schubert problem $\blambda$.
 \begin{enumerate}
  \item Compute all solutions $(H_1,\dotsc,H_{d(\blambda)})$ to a single instance of a Schubert
   problem for a regular value 
   $z=(\Fdot^1,\dotsc,\Fdot^r)$ of the map 
   $Y_{\blambda}\to Z_{\blambda}$.
  \item Use parameter homotopy to follow these $d(\blambda)$ solutions over a loop
    $\gamma\colon[0,1]\to Z_{\blambda}$ based at $z$ to compute a monodromy permutation.
 \end{enumerate}
Typically, (1) is quite challenging, while (2) is much easier

Leykin and Sottile~\cite{LS09} used this method to compute Galois groups of some Schubert
problems. 
For step (1), they implemented a simple version of the Pieri homotopy
algorithm~\cite{HSS98} to solve a single instance of a Schubert problem, then 
used off-the-shelf continuation software to compute monodromy permutations, and finally
called GAP~\cite{GAP4} to determine the group generated by these monodromy permutations.
In every simple Schubert problem studied the Galois group was the full
symmetric group.
We explain this in more detail.

A Schubert problem $\blambda=(\lambda^1,\dotsc , \lambda^r)$ is \demph{simple} if all but
at most two partitions $\lambda^i$ are equal to the partition $\I$, i.e.\ 
$\blambda = (\lambda^1,\lambda^2, \I,\dotsc,\I)$. 
Leykin and Sottile only looked at simple Schubert problems, and only tried to determine if
the monodromy was the full symmetric group.
The restriction to simple Schubert problems is because there was no efficient algorithm to
compute the numerical solutions to general Schubert problems, but the version of the Pieri
homotopy algorithm for simple Schubert problems is efficient and easy to implement.
Also, it is relatively easy to decide if a set of permutations
generates the full symmetric group and this algorithm has a fast implementation in GAP. 

Leykin and Sottile wrote a maple script\footnote{\footnotesize {\tt
 http://www.math.tamu.edu/\~{}sottile/research/stories/Galois/HoG.tgz}}
to study the Galois group of simple Schubert problems.
It first sets up and runs the Pieri homotopy algorithm to compute all solutions to a
general instance, calling PHCPack~\cite{V99} for path-continuation.
Then it starts computing monodromy permutations, again using PHCPack for
path-continuation.
When a new monodromy permutation is computed, it calls GAP to test if the  permutations 
computed so far generate the symmetric group.
If not, then it computes another monodromy permutation, and continues.

The largest Schubert problem studied was 
$(\TIs,\T,\I^{13})=17589$ in $\Gr(3,9)$.
Solving one instance and computing seven monodromy permutations took 78.2 hours (wall
time) on a single core.
The maximum number of monodromy permutations needed in any computation to
determine the full symmetric group was nine.

This verified that about two dozen simple Schubert problems have full symmetric Galois
group, including the problems $(\I^{k(n-k)})$ in $\Gr(k,n)$ for 
$k=2$ and $4\leq n\leq 10$, $k=3$ and $5\leq n\leq 8$, and $(k,n)$ equal to $(4,6)$ and
$(4,7)$, as well as three larger problems in $\Gr(3,9)$ and $\Gr(4,8)$.
Table~\ref{table:Num_Galois} records some data from the computations with $d(\blambda)>1000$.
%%%%%%%%%%%%%%%%%%%%%%%%%%%%%%%%%%%%%%%%%%%%%%%%%%%%%%%%%%%%%%%%%%%%%%%%
\begin{table}[htb]
 \caption{Galois group computation (h := hours)}\vspace{-10pt}
 \label{table:Num_Galois}

 {\begin{tabular}{|c||c|c|c|c|c|}\hline
  $k,n$ &2,10&3,8&3,9&3,9&4,8\\\hline
  problem&$\I^{16}$&$\I^{15}$&\rule{0pt}{11.5pt}$(\TIs^2,\I^{12})$
     &$(\TIs,\T,\I^{13})$&$(\TIs,\I^{13})$ \\\hline
  solutions&1430&6006&10329&17589&8580\\\hline
  time&2.6h&18.6h&49h&78.2h&44.5h\\\hline
  permutations&7&6&7&7&9\\\hline
 \end{tabular}}

\end{table}
%%%%%%%%%%%%%%%%%%%%%%%%%%%%%%%%%%%%%%%%%%%%%%%%%%%%%%%%%%%%%%%%%%%%%%%%%%%%%%%%%%%%%%
% [2, 4,     2],      :12     4       2
% [2, 5,     5],      :27     6       4
% [2, 6,    14],      :19     5       6
% [2, 7,    42],      :51     6       8
% [2, 8,   132],     4:14     7      10
% [2, 9,   429],    20:31     4      12
% [2, 10, 1430],  2:37:00     7      14
% [2, 11, 4862],  <--------------  Somehow this just cannot be computed
% [3, 5,     5],      :12     4       4
% [3, 6,    42],      :35     4       7
% [3, 7,   462],    17:54     5      10
% [3, 8,  6006],                     13
% [4, 6,    14],      :15     4       6
% [4, 7,   462],    23:34     5      10
% [3, 9, 210,210,   10329] 12   48:58:02   7  2 failures  short
% [3, 9, 210,200,   17589] 13  255        10  2 failures  long
%                               78:12      7              short
% [4, 8, 2100,1000,  8580] 12  44:29:51   9               short
%%%%%%%%%%%%%%%%%%%%%%%%%%%%%%%%%%%%%%%%%%%%%%%%%%%%%%%%%%%%%%%%%%%%%%%%%%%%%%%%%%%%%%
This suggests that the Galois group of any simple Schubert problem is the full
symmetric group on its set of solutions.
A first step towards this conjecture is Theorem~\ref{Th:SW_double}, which shows that
Galois groups of the simple Schubert Schubert problems $(\lambda,\I^{k(n-k)-|\lambda|})$ are
doubly transitive. 

These numerical methods which directly compute monodromy permutations give a second
approach to studying Galois groups of Schubert problems.
The main bottleneck is the lack of efficient algorithms to compute all
solutions to a single instance of a given Schubert problem.

In the same way that the geometric Pieri rule~\cite{So96} led to the efficient Pieri
homotopy algorithm~\cite{HSS98}, Vakil's geometric Littlewood-Richardson rule~\cite{Va06a}
leads to the efficient Littlewood-Richardson homotopy~\cite{SVV10} to 
solve any Schubert problem in a Grassmannian.
While this algorithm is proposed and described in~\cite{SVV10}, it lacks a practical 
implementation.
There is one being written in Macaulay 2~\cite{M2} based on Leykin's NAG4M2~\cite{NAG4M2}
package. 
When completed and optimized, our group plans an experiment along the lines proposed
in Subsection~\ref{S:Vakils_criterion} to use numerical algebraic geometry to compute Galois
groups of Schubert problems.
This is expected to be feasible for Schubert problems with up to $20000$ solutions with a
formulation with up to 25 local coordinates.

This work is affecting research in numerical algebraic geometry beyond the development
and implementation of the Littlewood-Richardson homotopy algorithm.
The method of regeneration~\cite{HSW11} may yield practical algorithms to compute
Schubert problems on other flag manifolds (Littlewood-Richardson homotopy is restricted to
the Grassmannian).
There are other possible continuation algorithms to develop and
implement.
We expect a broad numerical study of Galois groups of Schubert problems in other flag
manifolds to result from these investigations.

The most significant impact of~\cite{LS09} on numerical algebraic geometry is that it
has led to the incorporation of certification in software.
As mentioned, numerical algebraic geometry computes approximations to solutions of systems of
polynomial equations, and there are {\it a priori} no guarantees on the output.
Smale studied the convergence of Newton's method and developed \demph{$\alpha$-theory}, named
after a constant $\alpha$ that may be computed at a point
$x$ for a polynomial system $F$.
When $\alpha(x,F)\lesssim 0.15$, Newton iterations starting at $x$ are guaranteed to
converge quickly to a solution for $F=0$.
(This is explained in~\cite[Ch.~9]{BCSS}.)

Certification was recently  incorporated into software when Hauenstein and
Sottile released alphaCertified~\cite{alphaC}, which certifies the output of a numerical
solver. 
More fundamentally, Beltr{\'a}n and Leykin~\cite{BL12,BL13} extended $\alpha$-theory,
giving an algorithm for certified path-tracking which has certified that the
Schubert problem $\I^8=14$ in $\Gr(2,6)$ has Galois group equal to the full symmetric group
$S_{14}$.
Lastly, the (traditional) formulation of a Schubert problem in Section~\ref{SS:SC} typically
involves far more equations than variables, and $\alpha$-theory is only valid when the number
of equations is equal to the number of variables.
Hauenstein, Hein, and Sottile~\cite{HHS} have shown how to reformulate any Schubert problem
on a classical flag manifold as a system of $N$ bilinear equations in $N$ variables, 
enabling certification of general Schubert problems.

%%%%%%%%%%%%%%%%%%%%%%%%%%%%%%%%%%%%%%%%%%%%%%%%%%%%%%%%%%%%%%%%%%%%%%%%%%%%%%%%%
\subsection{Frobenius method and elimination theory}\label{S:Frobenious}
%%%%%%%%%%%%%%%%%%%%%%%%%%%%%%%%%%%%%%%%%%%%%%%%%%%%%%%%%%%%%%%%%%%%%%%%%%%%%%%%%

A third method to study Galois groups on a computer exploits symbolic computation and the
Chebotarev Density Theorem.
It is based on an observation which follows from Hilbert's Irreducibility Theorem
(see~\cite[p.~49]{BV}). 
Let $\blambda$ be a Schubert problem.
As $Z_{\blambda}=\prod_{i=1}^r \Fl(n)$
is a smooth rational variety, if $(\Fdot^1,\dotsc,\Fdot^r)$ is a
regular value of $Y_{\blambda}\to Z_{\blambda}$ with each flag defined over $\Q$, then the
smallest field of definition of the solutions to the corresponding instance
of $\blambda$,
 \begin{equation}\label{Eq:ISP}
  X_{\lambda^1}\Fdot^1\,\cap\,
  X_{\lambda^1}\Fdot^2\,\cap\; \dotsb\; \cap\,
  X_{\lambda^1}\Fdot^r\,,
 \end{equation}
is a finite extension of $\Q$ whose Galois group is a subgroup of $\calG_{\blambda}$.
These Galois groups coincide for a positive fraction of rational flags.

This gives a probabilistic method to determine $\calG_{\blambda}$.
Given a point $(\Fdot^1,\dotsc,\Fdot^r)\in Z_{\blambda}(\Q)$, formulate the
intersection~\eqref{Eq:ISP} as a system of polynomials and compute an eliminant, $g(x)$.
When $g$ is irreducible over $\Q$, the smallest field of definition of~\eqref{Eq:ISP} is
$\Q[x]/\langle g(x)\rangle$, and so its Galois group is the Galois group
\defcolor{$\calG_g$} of $g(x)$.  
Taking the largest group computed for different points in $Z_{\blambda}(\Q)$ 
determines $\calG_{\blambda}$ with high probability.

Unfortunately, this method is infeasible for the range of Schubert problems we are
interested in.
While it is possible to compute such eliminants $g(x)$ for a Schubert problem
$\blambda$ with  $d(\blambda)\lesssim 50$, we know of no software to
compute $\calG_g$ when $d(\blambda)\gtrsim 12$.
Also, the polynomials $g(x)$ are typically enormous with coefficients quotients of
$1000$-digit integers.

There is a feasible method that can either prove $\calG_g=S_{d(\blambda)}$
or give very strong information about $\calG_g$.
Suppose that $g(x)$ has integer coefficients.
Then for any prime $p$ that does not divide the discriminant of $g$, the reduction
of $g$ modulo $p$ is square-free.
The \demph{Frobenius automorphism} $\defcolor{F_p}\colon a\mapsto a^p$
acts on the roots of $g$, with one orbit for each
irreducible factor of $g$ in $\Z/p\Z[x]$. 
Thus the cycle type of $F_p$ is given by the degrees of the irreducible factors of $g$
in $\Z/p\Z[x]$. 

It turns out that $F_p$ lifts to a \demph{Frobenius element} in the characteristic zero
Galois group $\calG_g$ with the same cycle type.
Reducing $g(x)$ modulo different primes and factoring gives cycles types of many
elements of $\calG_g$.
By the Chebotarev Density Theorem, these Frobenius elements are distributed uniformly at
random in $\calG_g$, for $p$ sufficiently large.
This probabilistic method to understand the distribution of cycles types in
$\calG_g$ is often sufficient to determine $\calG_g$, as we explain below.

What makes this method practical is that
elimination commutes with reduction modulo $p$, computing eliminants modulo a prime
$p$ is feasible for $d(\blambda)\lesssim 500$, and factoring modulo a prime $p$ is also
fast. 
The primary reason for this efficiency is that when working modulo primes
$p<2^{64}$, arithmetic operations on coefficients take only one clock cycle.
A computation of a few hours in characteristic zero takes less than a second in 
characteristic $p$. 

The method we use to show that a Galois group $\calG_g$ is the full symmetric group
is based on a theorem of Jordan.
A subgroup $G$ of $S_n$ is \demph{primitive} if it preserves no non-trivial partition of
$\{1,\dotsc,n\}$. 

%%%%%%%%%%%%%%%%%%%%%%%%%%%%%%%%%%%%%%%%%%%%%%%%%%%%%%%%%%%%%%%%%%%%%%%%%%%%%%%%%
\begin{theorem}[Jordan~\cite{J1873}]
 If a primitive permutation group $G\subset S_n$ contains an $\ell$-cycle for some prime
 number $\ell<n{-}2$, then $G$ is either $S_n$ or its alternating subgroup $A_n$.
\end{theorem}
%%%%%%%%%%%%%%%%%%%%%%%%%%%%%%%%%%%%%%%%%%%%%%%%%%%%%%%%%%%%%%%%%%%%%%%%%%%%%%%%%

One way that a subgroup $G$ could be primitive would be if $G$ contains cycles of lengths
$n$ and $n{-}1$, for then it is doubly transitive.
Since one of $n$ or $n{-}1$ is even, $G$ is not a subgroup of $A_n$.
This gives the following algorithm to show that a Galois group is the full
symmetric group, which was suggested to us by Kiran Kedlaya.\medskip

\noindent{\bf Frobenius Algorithm.}
 Suppose that $\blambda$ is a Schubert problem.

\begin{enumerate}
 \item[0.] Set $\varepsilon_1 = \varepsilon_{2}=\varepsilon_3:=0$.

 \item[1.] Choose a sufficiently general point
           $(\Fdot^1,\dotsc,\Fdot^r)\in\Z_{\blambda}(\Q)$ and a prime $p$. 
            Compute an eliminant $g(x)\in \Z/p\Z[x]$ modulo $p$ for the corresponding
            instance of the Schubert problem $\blambda$.  

 \item[2.] Factor $g(x)$ modulo $p$,
\[
    g(x)\ =\ h_1(x)\dotsb h_s(x)\quad\mbox{in}\ \ \Z/p\Z[x]\,.
\]
 \begin{enumerate}
 \item[(i)] If $s=1$ so that $g$ is irreducible, set $\varepsilon_1:=1$.

 \item[(ii)] If $s=2$ and one of $h_1,h_2$ has degree $n{-}1$, set $\varepsilon_2:=1$.

 \item[(iii)] If one of the $h_i$ has degree a prime number $\ell$ with $n/2<\ell<n{-}2$,
   set $\varepsilon_3:=1$. 
\end{enumerate}

\item[3.] If $\varepsilon_1=\varepsilon_2=\varepsilon_3=1$ then proclaim
          ``$\calG_{\blambda}=S_2$'', otherwise return to step 1.

\end{enumerate}\medskip

Steps 2(i) and 2(ii) establish that $G_{\blambda}$ contains cycles of lengths $n$ and
$n{-}1$.
For Step 2(iii), the Frobenius element has a prime cycle of length $\ell$ and all other
cycles are shorter as $n/2<\ell$, so raising the Frobenius element to the power
$(\ell{-}1)!$ will result in an $\ell$-cycle in $\calG_{\blambda}$.

Assuming this samples elements of $\calG_{\blambda}$ uniformly at random (as appears to
be the case in practice), that $\calG_{\blambda}=S_n$, and that $n>6$, then 
2(i) will occur with probability $\frac{1}{n}$, 2(ii) with probability $\frac{1}{n-1}$,
and 2(iii) with the higher probability 
\[
   \sum_{n/2<\ell<n-2}\tfrac{1}{\ell}\ ,
\]
(the sum is over primes $\ell$).
In our experience, $95\%$ of the time the Frobenius algorithm took fewer than $2n$
steps. 

When $\calG_{\blambda}\neq S_n$, factoring eliminants modulo $p$ gives 
cycle types in $\calG_{\blambda}$ (and can give their distribution), which may be used to
help identify $\calG_{\blambda}$.
This method has been used for both these tasks, and it appears to be feasible for Schubert
problems with $d(\blambda)\lesssim 500$.

%%%%%%%%%%%%%%%%%%%%%%%%%%%%%%%%%%%%%%%%%%%%%%%%%%%%%%%%%%%%%%%%%%%%%%%%%%%%%%%%%
\subsection{Galois groups for $\Gr(4,8)$ and 
$\Gr(4,9)$}\label{S:search} 

We explain how the methods of Subsections~\ref{S:Vakils_criterion} and~\ref{S:Frobenious},
together with geometric arguments,  were used to nearly determine the Galois groups of all
Schubert problems in $\Gr(4,8)$ and $\Gr(4,9)$.

First, Vakil used his maple script based upon his combinatorial criteria (i) and (ii) to
show that all Schubert problems in $\Gr(2,n)$ for $n\leq 16$ and $\Gr(3,n)$ for $n\leq 9$
had at least alternating Galois groups.
By Grassmannian duality, the smallest Grassmannians not tested were $\Gr(4,8)$ and
$\Gr(4,9)$, which we tested (previously Vakil had tested $\Gr(4,8)$). 
We altered his program so that it only computed Schubert problems $\blambda$ for which 
no partition $\lambda$ in $\blambda$ had $\lambda_k>0$ or $\lambda_1=n{-}k$---these
immediately reduce to a smaller Grassmannian.
This program determined that 3468 of the 3501 such Schubert problems in $\Gr(4,8)$
have at least alternating Galois group.
For $\Gr(4,9)$, it determined that 36534 of the 36767 such Schubert problems have at least
alternating Galois group. 

Two of the remaining $33=3501-3468$ problems in $\Gr(4,8)$, $\I^{16}=24024$ and
$\TTs\cdot\I^{12}=2460$, were too large to compute modulo any prime $p$, but they are doubly
transitive by Theorem~\ref{Th:SW_double} and Vakil's Criterion (iii) implies that their Galois
groups are at least alternating.
Computing cycle types of Frobenius elements showed that 17 of the remaining 31
Schubert problems have full symmetric Galois groups, and implied that the remaining
fourteen had Galois group either $S_4$ (for Derksen's example $\TTs^4=6$) or $D_4$, and this
second type fell into two classes according to their underlying geometry, represented by 
 \begin{equation}\label{Eq:exceptional}
  \TTs^2\cdot \IIIs\cdot\Th\cdot \I^2\ =\ 4
   \qquad\mbox{and}\qquad
  \IIIs^2\cdot\Th^2\cdot \I^4\ =\ 4\,,
 \end{equation}
respectively.
We verified the predicted Galois groups for these fourteen using geometric arguments.
These three classes generalize to give infinite families of Schubert problems whose Galois
groups are not the full symmetric group.

We then applied the methods from Section~\ref{S:Frobenious} to study Galois groups of
most (209) of the $233=36767-36534$ problems in $\Gr(4,9)$ for which Vakil's criteria were
inconclusive. 
Computing cycles of Frobenius elements showed that 58 have full symmetric Galois groups.
The remaining 151 did not have full symmetric Galois groups and they fell into classes having
related geometry. 
All except one class is one of the families of classes from $\Gr(4,8)$.
For example, the problems
\[
   \IIIs^2\cdot\F^2\cdot\I^6\ =\ 5
   \quad\mbox{and}\quad
   \IIIs^2\cdot\F^2\cdot\T\cdot\I^4\ =\ 6
\]
are among the generalizations of the second problem in~\eqref{Eq:exceptional}, and do not
have full symmetric Galois groups.

Among the Schubert problems shown to have full symmetric Galois groups were the problems 
$\IIIs\cdot\Th\cdot\I^{10}=420$ and $\TTs^2\cdot\I^8=280$ in $\Gr(4,8)$, which gives
some idea of the size of Schubert problems that may be studied with this method.

Before giving two examples from $\Gr(4,8)$, we remark that Galois groups of Schubert
problems appear to either be highly transitive (e.g. full symmetric) or they act
imprimitively, failing  to be doubly transitive.

%%%%%%%%%%%%%%%%%%%%%%%%%%%%%%%%%%%%%%%%%%%%%%%%%%%%%%%%%%%%%%%%%%%%%%%%%%%%%%%%%
\subsubsection{The Schubert problem $\TTs^4=6$ in $\Gr(4,8)$.}\label{S:TT^4}

Derksen determined this Galois group.
While an instance of this problem is given by four complete flags in general position in
$\C^8$, only their 4-dimensional subspaces $L_1,\dotsc,L_4$ matter.
Its solutions will be those $H\in\Gr(4,8)$ for which $\dim H\cap L_i\geq 2$ for each
$i=1,\dotsc,4$.

To understand this problem, consider the \emph{auxiliary} problem $\Th^4$ in $\Gr(2,8)$
given by $L_1,\dotsc,L_4$.
This asks for the $h\in\Gr(2,8)$ with $\dim h\cap L_i\geq 1$ for $i=1,\dotsc,4$.
There are four solutions $h_1,\dotsc,h_4$ to this problem, and its Galois group is 
the full symmetric group $S_4$.
We also have that $h_1,\dotsc,h_4$ are in direct sum and they span $\C^8$.

Clearly, each of the four-dimensional linear subspaces $\defcolor{H_{i,j}}:=h_i\oplus h_j$ 
will meet eack $L_k$ in a 2-dimensional linear subspace, and so they are solutions to the
original problem. 
In fact, they are the only solutions.
It follows from this construction that the Galois group of $\TTs^4=6$ is $S_4$ acting on the
pairs $\{h_i,h_j\}$.
This is an imprimitive permutation group.

The structure of this problem shows that if the $L_i$ are real, then either two or
all six of the solutions will be real.
Indeed, if all four solutions $h_i$ to the auxiliary problem are real, than all six solutions
$H_{i,j}$ will also be real.
If however, two or four of the $h_i$ occur in complex conjugate pairs, then exactly two of
the $H_{i,j}$ will be real.

%%%%%%%%%%%%%%%%%%%%%%%%%%%%%%%%%%%%%%%%%%%%%%%%%%%%%%%%%%%%%%%%%%%%%%%%%%%%%%%%%
\subsubsection{The Schubert problem $\IIIs^2\cdot\Th^2\cdot\I^4=4$ in $\Gr(4,8)$.}
We will show that the Galois group of this problem is $D_4$, the group of symmetries of a 
square, which acts imprimitively on the solutions.

As above, an instance of this problem
is given by the choice of two 6-dimensional linear subspaces
$L_1,L_2$, two 2-dimensional linear subspaces $\ell_1,\ell_2$ and four 4-dimen\-sion\-al
linear subspaces $K_1,\dotsc,K_4$, all in general position.
Solutions will be those $H\in\Gr(4,8)$ such that
 \begin{equation}\label{Eq:SP}
   \dim H\cap L_i\ \geq\ 3\,,\ 
   \dim H\cap \ell_i\ \geq\ 1\,,\ \mbox{and}\ 
   \dim H\cap K_j\ \geq\ 1\,,\ 
 \end{equation}
for $i=1,2$ and $j=1,\dotsc,4$.

Consider the first four conditions in~\eqref{Eq:SP}.
Let $\defcolor{\Lambda}:=\langle \ell_1,\ell_2\rangle$, the linear span of $\ell_1$ and
$\ell_2$, which is isomorphic to $\C^4$.
Then $\defcolor{h}:=H\cap\Lambda$ is two-dimensional.
If we set $\ell_3:=\Lambda\cap L_1$ and $\ell_4:=\Lambda\cap L_2$, then 
$\dim h\cap \ell_3=\dim h\cap \ell_4=1$, and so 
$h\in\Gr(2,\Lambda)\simeq \Gr(2,4)$ meets each of the four two-planes
$\ell_1,\dotsc,\ell_4$.
In particular, $h$ is a solution to the instance of the problem of four lines
(realized in $\Gr(2,\Lambda)$) given by $\ell_1,\dotsc,\ell_4$, and therefore there are
two solutions, $h_1$ and $h_2$.

Now set $\defcolor{\Lambda'}:=L_1\cap L_2$, which is four-dimensional, and fix
one of the solutions \defcolor{$h_a$} to the problem of the previous paragraph.
For each $j=1,\dotsc,4$, set $\defcolor{\mu_j}:=\langle h_a, K_j\rangle\cap\Lambda'$,
which is two-dimensional.
These four two-planes are in general position and therefore give a problem of four lines
in $\P(\Lambda')$.
Let \defcolor{$\eta_{a,1}$} and \defcolor{$\eta_{a,2}$} be the two solutions to this
problem, so that $\dim \eta_{a,b}\cap\mu_j\geq 1$ for each $j$.

Then the four subspaces $\defcolor{H_{ab}}:=\langle h_a,\eta_{a,b}\rangle$ are solutions to
the original Schubert problem.
Indeed, since $\dim H_{ab}\cap\ell_j=1$ for $j=1,\dotsc,4$ and 
$\eta_{a,b}\subset L_1\cap L_2$, we have $\dim H_{ab}\cap L_i=3$, and so $H_{ab}$
satisfies the first four conditions of~\eqref{Eq:SP}. 
Since $\Lambda\cap\Lambda'=\{0\}$, $h_a$ does not meet $\mu_j$ for $j=1,\dotsc,4$, so 
$\dim H_{ab}\cap\langle h_a,K_j\rangle=3$, which implies that $\dim H_{ab}\cap K_j\geq 1$,
and shows that $H_{ab}$ is a solution.

The Galois group $\Gal{\blambda}$ acts imprimitively as it preserves the partition 
$\{H_{11},H_{12}\}\sqcup\{H_{21},H_{22}\}$ of the solutions.
Since $\Gal{\blambda}$ is transitive, it is either $\Z_2\times\Z_2$, which has order four,
or the dihedral group $D_4$, which has order eight.
Computing cycle types of Frobenius elements shows that $\Gal{\blambda}$
contains cycles of type $(4)$, $(2,2)$, $(2,1,1)$, and $(1,1,1,1)$, which shows that it is 
$D_4$.
Computing $100000$ eliminants modulo $11311$, $99909$ were square free and therefore
gave Frobenius elements in $\Gal{\blambda}$.
We summarize the observed frequencies of each cycle type in the following table.
\[
  \begin{tabular}{|l||c|c|c|c|}\hline
   {cycle type}&$(4)$&$(2,2)$&$(2,1,1)$&$(1,1,1,1)$\\\hline
   {frequency}&$25014$&$37384$&$25145$&$12366$\\\hline
   {fraction}&$0.2504$&$0.3742$&$0.2517$&$0.1238$\\\hline  
  \end{tabular}
\]
%>  Freq[[4]]:=25014:                   0.2503678347
%>  Freq[[2, 2]]:=37384:                0.3741805043
%>  Freq[[2, 1, 1]]:=25145:             0.2516790279
%>  Freq[[1, 1, 1, 1]]:=12366:          0.1237726331
% 25014+37384+25145+12366;

%
%result_id=7
We close with the observation that when we computed real osculating instances of this
Schubert problem (as part of the experiment described in Section~\ref{S:real_structure}),
we only found either zero or four real solutions, and never two. 
In all other Schubert problems with imprimitive Galois groups that we computed, we found
similar interesting structure in their numbers of real solutions.
%%%%%%%%%%%%%%%%%%%%%%%%%%%%%%%%%%%%%%%%%%%%%%%%%%%%%%%%%%%%%%%%%%%%%%%%%%

\providecommand{\bysame}{\leavevmode\hbox to3em{\hrulefill}\thinspace}
\providecommand{\MR}{\relax\ifhmode\unskip\space\fi MR }
\providecommand{\MRhref}[2]{%
  \href{http://www.ams.org/mathscinet-getitem?mr=#1}{#2}
}
\providecommand{\href}[2]{#2}

\end{document}